\documentclass[11pt]{amsart}
\usepackage{amsfonts}
\usepackage{amsmath}
\usepackage{amssymb,latexsym}
\usepackage{graphicx}
\usepackage{titletoc}
\usepackage{amssymb}
\usepackage{color}

\providecommand{\U}[1]{\protect \rule{.1in}{.1in}}
\newtheorem{theorem}{Theorem}[section]

\newtheorem{corollary}[theorem]{Corollary}

\newtheorem{definition}[theorem]{Definition}

\newtheorem{lemma}[theorem]{Lemma}

\newtheorem{proposition}[theorem]{Proposition}
\theoremstyle{remark}
\newtheorem{remark}[theorem]{Remark}

\numberwithin{equation}{section}
\begin{document}
\title{The geometry of generalized Lam\'{e} equation, I}
\author{Zhijie Chen}
\address{Department of Mathematical Sciences, Yau Mathematical Sciences
Center, Tsinghua University, Beijing, 100084, China }
\email{zjchen2016@tsinghua.edu.cn}
\author{Ting-Jung Kuo}
\address{Department of Mathematics, National Taiwan Normal University, Taipei 11677, Taiwan }
\email{tjkuo1215@ntnu.edu.tw, tjkuo1215@gmail.com}
\author{Chang-Shou Lin}
\address{Taida Institute for Mathematical Sciences (TIMS), Center for
Advanced Study in Theoretical Sciences (CASTS), National Taiwan University,
Taipei 10617, Taiwan }
\email{cslin@math.ntu.edu.tw}

\begin{abstract}
In this paper, we prove that the spectral curve $\Gamma_{\mathbf{n}}$ of the
generalized Lam\'{e} equation with the Treibich-Verdier potential
\begin{equation*}
y^{\prime \prime }(z)=\bigg[  \sum_{k=0}^{3}n_{k}(n_{k}+1)\wp(z+\tfrac{%
\omega_{k}}{2}|\tau)+B\bigg]  y(z),\text{ \ }n_{k}\in \mathbb{Z}_{\geq0}
\end{equation*}
can be embedded into the symmetric space Sym$^{N}E_{\tau}$ of the $N$-th copy of
the torus $E_{\tau}$, where $N=\sum n_{k}$. This embedding induces an
addition map $\sigma_{\mathbf{n}}(\cdot|\tau)$ from $\Gamma_{\mathbf{n}}$
onto $E_{\tau}$. The main result is to prove that the degree of $\sigma _{%
\mathbf{n}}(\cdot|\tau)$ is equal to%
\begin{equation*}
\sum_{k=0}^{3}n_{k}(n_{k}+1)/2.
\end{equation*}
This is the first step toward constructing the premodular form associated
with this generalized Lam\'{e} equation.
\end{abstract}

\maketitle

\section{Introduction}

Throughout the paper, we let $\tau\in\mathbb{H}=\{\tau\in\mathbb{C}|\operatorname{Im}%
\tau>0\}$, $E_\tau=\mathbb{C}/\Lambda_\tau$ be a flat torus with the lattice
$\Lambda_\tau=\mathbb{Z}+ \mathbb{Z}\tau$, and $\wp(z)=\wp(z|\tau)$ be the
Weierstrass elliptic function with periods $\omega_1=1$, $\omega_2=\tau$ and
$\omega_3=1+\tau$. Let $\zeta(z)=\zeta(z|\tau):=-\int^{z}\wp(\xi|\tau)d\xi$
be the Weierstrass zeta function with two quasi-periods $\eta_{j}(\tau)$, $%
j=1,2$:%
\begin{equation}
\eta_{j}(\tau)=\zeta(z+\omega_{j} |\tau)-\zeta(z|\tau),\quad j=1,2,
\label{40-2}
\end{equation}
and $\sigma(z)=\sigma(z|\tau)$ be the Weierstrass sigma function defined by $%
\sigma(z):=\exp \int^{z}\zeta(\xi)d\xi$. Notice that $\zeta(z)$ is an odd
meromorphic function with simple poles at $\Lambda_{\tau}$ and $\sigma(z)$
is an odd entire function with simple zeros at $\Lambda_{\tau}$. For $z\in%
\mathbb{C}$ we denote $[z]:=z \ (\text{mod}\ \Lambda_{\tau}) \in E_{\tau}$.
For a point $[z]$ in $E_{\tau}$ we often write $z$ instead of $[z]$ to
simply notations when no confusion arises.

In this paper, we consider the complex second order ODE:
\begin{equation}  \label{eq21}
y^{\prime \prime }(z)=(I_{\mathbf{n}}(z;\tau)+B)y(z),\quad z\in\mathbb{C}.
\end{equation}
where $B\in\mathbb{C}$ and
\begin{equation}
I_{\mathbf{n}}(z;\tau):=\sum_{k=0}^3n_k(n_k+1)\wp(z+\tfrac{\omega_k}{2}|\tau)
\end{equation}
with $\omega_0=0$ and $n_k\in \mathbb{Z}_{\geq 0}$ for all $k$. We also
denote $I_{\mathbf{n}}(z;B,\tau):=I_{\mathbf{n}}(z;\tau)+B$. The $I_{\mathbf{%
n}}(z;\tau)$ is called the \emph{Treibich-Verdier potential}. In \cite{T,TV}
Treibich and Verdier proved that $I_{\mathbf{n}}(z;\tau)$ is an algebro-geometric
solution of the KdV hierarchy equations or equivalently a
finite-gap potential. Later Gesztesy and Weikard \cite{GW} generalized their
result to prove that any Picard potential is an algebro-geometric solution
of the KdV hierarchy equations.

We briefly recall the notion of algebro-geometric solutions. Let $y_1(z;B)$,
$y_2(z;B)$ be two solutions of (\ref{eq21}) and $\Phi(z;B):=y_1(z;B)y_2(z;B).
$ Then a direct computation show that $\Phi$ satisfies the following third
order ODE (called the second symmetric product equation of (\ref{eq21})):
\begin{equation}  \label{eq23}
\Phi^{\prime \prime \prime }(z;B)-4(I_{\mathbf{n}}(z;\tau)+B)\Phi^{\prime
}(z;B)-2I_{\mathbf{n}}^{\prime }(z;\tau)\Phi(z;B)=0.
\end{equation}
Multiplying $\Phi$ and integrating \eqref{eq23} ,we obtain that
\begin{equation}
\Phi{^{\prime }}(z;B)^2-2\Phi(z;B)\Phi^{\prime \prime }(z;B)+4(I_{\mathbf{n}%
}(z;\tau)+B)\Phi(z;B)^2  \label{eq24}
\end{equation}
is independent of $z$. Let $Q_{\mathbf{n}}(B)=Q_{\mathbf{n}}(B;\tau)$ denote
the above expression of \eqref{eq24}. Then $I_{\mathbf{n}}(z;\tau)$ is an
algebro-geometric solution of the KdV hierarchy equations if $Q_{\mathbf{n}%
}(B)$ is a polynomial of $B$ for some solution $\Phi(z;B)$; see \cite{GW}.
In this case, $Q_{\mathbf{n}}(B)$ is known as the \emph{spectral polynomial}
and $\Gamma_{\mathbf{n}}=\Gamma_{\mathbf{n}}(\tau):=\{(B,W)|W^2=Q_{\mathbf{n}}(B;\tau)\}$ is called the
\emph{spectral curve} of the potential $I_{\mathbf{n}}(z;\tau)$.

When $\mathbf{n}=(n,0,0,0)$, the potential $n(n+1)\wp(z|\tau)$ is called the
Lam\'{e} potential and
\begin{equation}  \label{Lame}
y^{\prime \prime }(z)=(n(n+1)\wp(z|\tau)+B)y(z),\quad z\in\mathbb{C}
\end{equation}
is called the Lam\'{e} equation. The fact that the Lam\'{e} potential is a
finite-gap potential was first discovered by Ince \cite{Ince}. We refer the
readers to the classic texts \cite{Halphen,Poole,Whittaker-Watson} and recent works \cite{Beukers-Waall,CLW,Dahmen,Maier} for the
Lam\'{e} equation. Therefore, \eqref{eq21} is called a \emph{generalized Lam%
\'{e} equation (GLE)} in this paper. See \cite{CKLT,GW1,Takemura1,Takemura2,Takemura3,Takemura4,Takemura5} and references therein for recent
developments of GLE \eqref{eq21}.

In this paper, we want to study \eqref{eq21} from the aspect of monodromy
representation. Clearly, \eqref{eq21} can be described as a Fuchsian
equation defined on the torus $E_\tau$ with four regular singularities at $%
\frac{\omega_k}{2}$'s. The local exponents, i.e. the roots of its indicial
equation at $\frac{\omega_k}{2}$ are $-n_k$, $n_{k}+1$. Since $n_k\in
\mathbb{Z}$, it is well-known (cf. \cite{GW1}) that $I_{\mathbf{n}}(z)$ is a Picard
potential, i.e. any solution of \eqref{eq21} is meromorphic in $\mathbb{C}$.
By applying $x=\wp(z)$, \eqref{eq21} can be transformed a second order ODE
on $\mathbb{CP}^1$ with four regular singular points at $e_k$'s and $\infty$%
, where $e_k:=\wp(\frac{\omega_k}{2}), k=1,2,3$. This is the well-known Heun
equation with four singular points, which has been extensively studied since
its isomonodromic deformation give arise to the famous Painlev\'{e} VI
equation; see e.g. \cite{GP}. For GLE \eqref{eq21}, we proved in \cite%
{Chen-Kuo-Lin} that the corresponding isomonodromic deformation equation is
the elliptic form of the Painlev\'{e} VI equation. In this paper, we also
denote GLE \eqref{eq21} by $H(\mathbf{n},B,\tau)$.

The monodromy representation $\rho_\tau$ of \eqref{eq21} is a group homomorphism
from $\pi_1(E_\tau)$ to $SL(2,\mathbb{C})$ because $I_{\mathbf{n}}(z;\tau)$
is a Picard potential. Since $\pi_1(E_\tau)$ is abelian, the monodromy group
is always abelian. Thus the monodromy group is comparely easier to compute
for \eqref{eq21} on $E_\tau$ than the Heun equation on $\mathbb{CP}^1$. In
terms of any linearly independent solutions $y_1(z)$ and $y_2(z)$, the
monodromy group is generated by two matrices $M_1$, $M_2\in SL(2,\mathbb{C})$
satisfying
\begin{equation}
(y_1,y_2)(z+\omega_i)=(y_1(z),y_2(z))M_i,\; i=1,2,\;\text{and }\;
M_1M_2=M_2M_1.  \label{15}
\end{equation}
By \eqref{15}, $M_1$ and $M_2$ can be normalized to satisfy one
of the followings.

\begin{enumerate}
\item[a)] If $\rho_\tau$ is completely reducible, then
\begin{equation}  \label{16}
M_1=\left(%
\begin{matrix}
e^{-2\pi is} & 0 \\
0 & e^{2\pi is}%
\end{matrix}%
\right),M_2=\left(%
\begin{matrix}
e^{2\pi ir} & 0 \\
0 & e^{-2\pi ir}%
\end{matrix}%
\right),\;\;(r,s)\in\mathbb{C}^2.
\end{equation}
See Section 2, where we will see that $(r,s)\notin\frac{1}{2}\mathbb{Z}^2$.

\item[b)] If $\rho_\tau$ is not completely reducible, then
\begin{equation}  \label{16-7}
M_1=\left(%
\begin{matrix}
1 & 0 \\
1 & 1%
\end{matrix}%
\right),M_2=\left(%
\begin{matrix}
1 & 0 \\
C & 1%
\end{matrix}%
\right),\;\;C\in\mathbb{C}\cup\{\infty\}.
\end{equation}
When $C=\infty$, the monodromy matrices are understood as
\begin{equation}
M_1=\left(%
\begin{matrix}
1 & 0 \\
0 & 1%
\end{matrix}%
\right),M_2=\left(%
\begin{matrix}
1 & 0 \\
1 & 1%
\end{matrix}%
\right).
\end{equation}
\end{enumerate}

The aformentioned spectral polynomial $Q_{\mathbf{n}}(B;\tau)$ also plays an
important role for the monodromy representation: $\rho_\tau$\textit{\ is
completely reducible if and only if $Q_{\mathbf{n}}(B;\tau)\neq 0$}.
Obviously, not all $2\times 2$ matrices of the form (\ref{16})-(\ref{16-7})
are monodromy matrices of \eqref{eq21}. Thus the following questions
naturally arise:

\begin{enumerate}
\item[(1)] If $Q_{\mathbf{n}}(B;\tau)\neq 0$, how to determine the monodromy
data $(r,s)$?

\item[(2)] If $Q_{\mathbf{n}}(B;\tau)=0$, how to determine the monodromy
data $C$?
\end{enumerate}

For the Lam\'{e} equation (\ref{Lame}), in \cite{CLW,LW,LW2} Chai, Wang and the
third author have constructed a \emph{premodular form} $Z^n_{r,s}(\tau)$
such that the monodromy matrices $M_1$, $M_2$ of (\ref{Lame}) at $\tau=\tau_0
$ with some $B$ are given by \eqref{16} if and only if $Z^n_{r,s}(\tau_0)=0$%
. Therefore, the image of $M_1$, $M_2$ for $\rho_{\tau_0}$ is $\{(r,s)\in%
\mathbb{C}^2\setminus\frac{1}{2}\mathbb{Z}^2|Z^n_{r,s}(\tau_0)=0\}$. We note
that $Z^n_{r,s}(\tau)$ is holomorphic in $\tau$ if $(r,s)\in\mathbb{R}%
^2\setminus \frac{1}{2}\mathbb{Z}^2$. Moreover, $Z^n_{r,s}(\tau)$ is a
modular form of weight $\frac{n(n+1)}{2}$ w.r.t. the principal congruence
subgroup $\Gamma(N)$ if $(r,s)$ is a $N$-torsion point; see \cite{LW2}. Thus
$Z^n_{r,s}(\tau)$ is called a \emph{premodular form}.

In this paper and the subsequent one \cite{CKL-pre}, we want to extend the
result in \cite{LW2} to include the Trebich-Verdier potential. Precisely, we
will establish the following theorem in \cite{CKL-pre}:

\begin{theorem}
\label{thm-premodular}  There exists a premodular form $Z_{r,s}^{\mathbf{n}%
}(\tau)$ defined in $\tau\in\mathbb{H}$ for any pair of $(r,s)\in\mathbb{C}%
^2\setminus \frac{1}{2}\mathbb{Z}^2$ such that the followings hold.

\begin{enumerate}
\item[(a)] If $(r,s)=(\frac{k_1}{N},\frac{k_2}{N})$ with $N\in 2\mathbb{N}%
_{\geq 2}$, $k_1,k_2\in\mathbb{Z}_{\geq 0}$ and $\gcd(k_1,k_2,N)=1$, then $%
Z_{r,s}^{\mathbf{n}}(\tau)$ is a modular form of weight $%
\sum_{k=0}^3n_k(n_k+1)/2$ with respect to the principal congruence subgroup $%
\Gamma(N)$.

\item[(b)] For $(r,s)\in\mathbb{C}^2\setminus\frac{1}{2}\mathbb{Z}^2$, $%
Z_{r,s}^{\mathbf{n}}(\tau_0)=0$ for some $\tau_0\in\mathbb{H}$ if and only
if there is $B\in\mathbb{C}$ such that GLE \eqref{eq21} with $\tau=\tau_0$
has its monodromy matrices $M_1$ and $M_2$ given by \eqref{16}.
\end{enumerate}
\end{theorem}

Following the ideas in \cite{CLW,LW2}, the spectral curve $\Gamma_{\mathbf{n}%
}(\tau)$ can be embedded into Sym$^N E_{\tau}:=E_{\tau}^N/S_N$, the
symmetric space of $N$-th copy of $E_\tau$, where $N:=\sum_{k=0}^3n_k$.%
\footnote{%
This $N$ has no relation with that in Theorem \ref{thm-premodular}-(a).}
Obviously, Sym$^NE_\tau$ has a natural addition map to $E_\tau$: $%
\{a_1,\cdots,a_N\}\mapsto\sum_{i=1}^Na_i$. Then the composition give arise to
a finite morphism $\sigma_{\mathbf{n}}(\cdot|\tau):\overline{\Gamma_{\mathbf{%
n}}(\tau)}\rightarrow E_\tau$, still called the \emph{addition map}. The
degree of $\sigma_{\mathbf{n}}$ is defined as $\deg\sigma_{\mathbf{n}%
}(\cdot|\tau)=\#\sigma_{\mathbf{n}}^{-1}(z)$, $z\in E_\tau$, counted with
multiplicity. Our main theorem in this paper is

\begin{theorem}[=Theorem \ref{degreeformula}]
\label{thm-degree} Let $\tau\in\mathbb{H}$. Then the addition map $\sigma_{%
\mathbf{n}}(\cdot|\tau): \overline{\Gamma_{\mathbf{n}}(\tau)} \to E_\tau$
has degree $\sum_{k=0}^3n_k(n_k+1)/2$.
\end{theorem}

A corollary of Theorem \ref{thm-degree} is that $\deg\sigma_{\mathbf{n}%
}(\cdot|\tau)$ (the same as the weight of the premodular form in Theorem \ref%
{thm-premodular}) is independent of $\tau$, which is not very obvious at the
moment. For the case of the Lam\'{e} equation, Theorem \ref{thm-degree} was
proved in \cite{LW2} by applying \emph{Theorem of the Cube} for morphisms
between varieties in algebraic geometry. But this method seems not work in
the general case. Our strategy is to study the general class of ODE:
\begin{equation}
y^{\prime \prime}(z)=I_{\mathbf{n}}(z;p,A,\tau)y(z),   \label{1505}
\end{equation}
where the potential $I_{\mathbf{n}}(z;p,A,\tau)$ is given by%
\begin{equation}
I_{\mathbf{n}}(z;p,A,\tau) =\left[
\begin{array}{l}
\sum_{k=0}^{3}n_{k}(n_{k}+1)\wp( z+\tfrac{\omega_{k}}{2}|\tau) +\frac{3}{4}%
(\wp(z+p|\tau) \\
+\wp(z-p|\tau))+A( \zeta(z+p|\tau)-\zeta(z-p|\tau)) +B%
\end{array}
\right] ,   \label{1503}
\end{equation}
with $A\in \mathbb{C}$, $p\in E_{\tau}\setminus E_{\tau}[2]$, $%
E_{\tau}[2]:=\{\frac{\omega_k}{2}|k=0,1,2,3\}+\Lambda_{\tau}$ and
\begin{equation}
B=A^{2}-\zeta(2p) A-\frac{3}{4}\wp ( 2p) -\sum _{k=0}^{3}n_{k}( n_{k}+1) \wp
( p+\tfrac{\omega_{k}}{2}|\tau).   \label{1101}
\end{equation}
The identity \eqref{1101} is to guarantee that all the singular points of %
\eqref{1505} are apparent i.e. all solutions of \eqref{1505} are free of
logarithmic singularity at any singular point. See \cite{Chen-Kuo-Lin,CKL-JGP,CKLW,Takemura} for recent developments of (\ref{1505}). Like \eqref{eq21}, we could
associate a hyperelliptic curve $\Gamma_{\mathbf{n},p}(\tau):=\{(A,W)|W^2=Q_{%
\mathbf{n},p}(A;\tau)\}$ and an addition map $\sigma_{\mathbf{n},p}$ with %
\eqref{1505}. These will be established in Sections 2 and 3. The reason we
introduce \eqref{1505} is that as $p\rightarrow\omega_k/2,k=0,1,2,3$, the
limiting equation of \eqref{1505} would be \eqref{eq21} with $\mathbf{n}=%
\mathbf{n}^\pm_k$, where $\mathbf{n}^\pm_k$ is defined by replacing $n_k$ in
\textbf{n} with $n_k\pm 1$. Due to this relation with \eqref{eq21}, we
expect to have the following connection.

\begin{theorem}[=Theorem \ref{main thm}]
\label{thm-degree-addition} For $k\in \{0,1,2,3\}$, there holds
\begin{equation*}
\deg\sigma_{\mathbf{n},p}(\cdot|\tau)=\deg\sigma_{\mathbf{n}%
^+_k}(\cdot|\tau)+\deg\sigma_{\mathbf{n}^-_k}(\cdot|\tau).
\end{equation*}
\end{theorem}

The paper is organized as follows. In Section \ref{monodromy}, we will give
a brief review of the monodromy representation of \eqref{1505}. The similar
argument also holds for \eqref{eq21}. We will prove the existence of the
embedding of $\Gamma_{\mathbf{n},p}$ and the addition map from $\Gamma_{%
\mathbf{n},p}$ onto $E_\tau$ in Section 3. In Sections 4 and 5, we will
study the limiting problem of \eqref{1505} under two case: (i) fix $p$ and $%
A\rightarrow\infty$; (ii) $p\rightarrow\frac{\omega_k}{2}$ and $%
A(p)\rightarrow\infty$. This limit problem plays a crucial role in our study
of the degree in Section 6, where Theorem \ref{thm-degree-addition} is
proved and then Theorem \ref{thm-degree} will be obtained by applying Theorem %
\ref{thm-degree-addition}.

\section{ Monodromy representation}

\label{monodromy}

Let $\mathbf{n}=(n_{0},n_{1},n_{2},n_{3})$, $n_{k}\in \mathbb{C}$. For any
fixed $\tau \in \mathbb{H}$ and $\pm p\not \in E_{\tau}[2]:=\{\frac{\omega_k}{2}|k=0,1,2,3\}+\Lambda_{\tau}$, we consider the
following generalized Lam\'{e} equation GLE$(\mathbf{n},p,A,\tau)$ on the torus $%
E_{\tau}$:%
\begin{equation}
y^{\prime \prime}(z)=I_{\mathbf{n}}(z;p,A,\tau)y(z)\text{ \ on }E_{\tau},
\label{5051}
\end{equation}
where the potential $I_{\mathbf{n}}(z;p,A,\tau)$ is given by%
\begin{equation}
I_{\mathbf{n}}(z;p,A,\tau) =\left[
\begin{array}{l}
\sum_{k=0}^{3}n_{k}(n_{k}+1)\wp( z+\tfrac{\omega_{k}}{2}|\tau) +\frac{3}{4}%
(\wp(z+p|\tau) \\
+\wp(z-p|\tau))+A(\zeta(z+p|\tau)-\zeta(z-p|\tau)) +B%
\end{array}
\right] ,   \label{5031}
\end{equation}
with $A\in \mathbb{C}$ and
\begin{equation}
B=A^{2}-\zeta( 2p) A-\frac{3}{4}\wp ( 2p) -\sum _{k=0}^{3}n_{k}( n_{k}+1)
\wp ( p+\tfrac{\omega_{k}}{2}|\tau).   \label{101}
\end{equation}
GLE$(\mathbf{n},p,A,\tau)$ is of Fuchsian type with singularities at $%
S:=E_{\tau}[2]\cup \{ \pm \lbrack p]\}$.

Let us briefly recall the associated monodromy representation. Let $Y(z;\tau)
$ $=(y_{1}( z;\tau ) ,y_{2}(z;\tau))$ be a fundamental system of solutions
of GLE$(\mathbf{n},p,A,\tau)$ near a fixed base point $q_{0}\not \in S$. In
general, $Y(z;\tau)$ is multi-valued with respect to $z$ and might have
branch points at $S$. For any loop $\ell$ $\in \pi_{1}(E_{\tau}\backslash
S,q_{0})$, there exists a matrix $\rho_{\tau }(\ell)\in SL(2,\mathbb{Z})$
such that $\ell^{\ast}Y(z;\tau)=Y(z;\tau)\rho_{\tau}(\ell)$. Here $%
\ell^{\ast }Y(z;\tau)$ denotes the analytic continuation of $Y(z;\tau)$
along the loop $\ell$. This induces a group homomorphism
\begin{equation}
\rho_{\tau}:\pi_{1}(E_{\tau}\backslash S,q_{0})\rightarrow SL(2,\mathbb{C}).
\label{mr}
\end{equation}
which is called the monodromy representation of the GLE$(\mathbf{n},p,A,\tau)
$.

The local exponent of GLE$(\mathbf{n},p,A,\tau)$ at $\frac{\omega_{k}}{2}$
are $-n_k, n_k+1$. In this paper, we are interested in the case $n_{k}\in
\mathbb{Z}_{\geq0}$ (hereafter we always assume $n_{k}\in \mathbb{Z}_{\geq0}$
for all $k$), because in this case the local monodromy matrix at each $\frac{%
\omega_{k}}{2}$ is $I_{2}$ and then the monodromy representation can be
reduced to a homomorphism $\rho_{\tau}:$ $\pi_{1}(E_{\tau}\backslash \{\pm
[p]\},q_{0})\to SL(2,\mathbb{Z})$. Let $\gamma_{\pm}\in\pi_{1}(E_{\tau}%
\backslash S, q_{0})$ be a simple loop encircling $\pm p$ counterclockwise
respectively, and $\ell_{j}$, $j=1,2$, be two fundamental cycles of $E_{\tau}
$ connecting $q_{0}$ with $q_{0}+\omega_{j}$ such that $\ell_{j}$ does not
intersect with $\ell_{p}+\Lambda_{\tau}$ (here $\ell_{p}$ is the straight
segment connecting $\pm p$) and satisfies
\begin{equation}
\gamma_{+}\gamma_{-}=\ell_{1}\ell_{2}\ell_{1}^{-1}\ell_{2}^{-1}\text{ in }%
\pi_{1}\left( E_{\tau}\backslash  \{ \pm [p] \} ,q_{0}\right) .
\label{II-iv}
\end{equation}
On the other hand, the local exponents at $\pm p$ are $\frac{-1}{2}, \frac{3%
}{2}$. Since \eqref{101} implies that $\pm p$ are apparent singularities
(i.e. non-logarithmic, see \cite{Chen-Kuo-Lin}), we have
\begin{equation}
\rho_{\tau}(\gamma_{\pm})=-I_{2}.   \label{89-2}
\end{equation}
Denote $\rho_{\tau}(\ell_{j})$ by $M_{j}$. Then the monodromy group of GLE$(%
\mathbf{n},p,A,\tau)$ is generated by $\{-I_{2},M_{1},M_{2}\}$. Together
with (\ref{II-iv})-(\ref{89-2}), we immediately obtain $M_{1}M_{2}=M_{2}M_{1}
$, which implies that there is always a solution denoted by $%
y_1(z)=y_{1}(z;A)$ being a common eigenfunction, i.e. $\ell^{%
\ast}_jy_{1}(z;A)=\varepsilon_j y_{1}(z;A)$, $j=1,2$. Therefore the
monodromy representation is always \textit{reducible}.

From the local exponents at $\frac{\omega_{k}}{2}$ and $\pm p$, it is easy
to see that $\frac{\omega_{k}}{2}$ is not a branch point of $y_{1}(z)$ but $%
\pm p$ is a branch point with ramification index $2$, i.e. $y_{1}(\pm
p+e^{2\pi i}z)=-y_{1}(\pm p+z)$ if $| z|>0$ is small. Then $y_{1}(z)$ can be
viewed as a single-valued meromorphic function in $\mathbb{C}%
\backslash(\ell_{p}+\Lambda_{\tau})$. Therefore, the analytic continuation
of $y_1(z)$ along the fundamental cycles $\ell_j$ is the translation of $y_1$
by $\omega_j,j=1,2$, namely
\begin{align}
y_{1}(z+\omega_{j};A) =\ell_{j}^{\ast}y_{1}(z;A)
=\varepsilon_{j}y_{1}(z;A),\quad j=1,2.  \label{304-111}
\end{align}

Since $\mathbb{C}\backslash(\ell_p+\Lambda_\tau)$ is symmetric about 0, $%
y_1(-z;A)$ is well-defined in $\mathbb{C}\setminus(\ell_p+\Lambda_\tau)$ and
also a solution of the same GLE$(\mathbf{n},p,A,\tau)$. Define
\begin{equation}
y_2(z)=y_2(z;A):= y_1(-z;A) \quad\text{in}\;\,\mathbb{C}\backslash(\ell_p+%
\Lambda_\tau).
\end{equation}
Then by \eqref{304-111} we see that $y_2(z;A)$ is also a common
eigenfunction, i.e.
\begin{equation}
y_{2}(z+\omega_{j};A)=\ell_{j}^{\ast}y_{2}(z;A)=%
\varepsilon_{j}^{-1}y_{2}(z;A),\text{ }j=1,2.   \label{304-12}
\end{equation}
If $\varepsilon_j\neq \pm 1$ for some $j$, then $y_1(z)$ and $y_2(z)$ are
linearly independent. In general, $y_1(z)$ and $y_2(z)$ might be linearly
dependent.

\begin{definition}
GLE$(\mathbf{n},p,A,\tau)$ is called \textit{completely reducible} if its
monodromy group acting on the 2-dimensional solution space has two linearly
independent common eigenfunctions. Otherwise, it is called not completely
reducible.
\end{definition}

We will see later that \textit{\ GLE$(\mathbf{n},p,A,\tau)$ is completely
reducible if and only if the above $y_1(z)$ and $y_2(z):= y_1(-z)$ are
linearly independent.}

The branch point $\pm p$ of $y_1(z)$ might cause trouble in analysis. To
avoid it, we introduce
\begin{equation}
\Psi_{p}(z):=\frac{\sigma(z) }{\sqrt{\sigma(z-p) \sigma (z+p)}}.
\label{61-3}
\end{equation}
By using the transformation law of $\sigma(z)$,
\begin{equation}  \label{123123}
\sigma(z+\omega_j)=-e^{\eta_j(z+\frac{1}{2}\omega_j)}\sigma(z),\quad j=1,2,
\end{equation}
we see that $\Psi_p(z)^2$ is an \emph{elliptic function}. We have the
following lemma.

\begin{lemma}
\label{lem2-6}Recall that $\ell_{j},j=1,2$ are the two fixed fundamental
cycles of $E_{\tau}$ which do not intersect with $\ell_{p}+\Lambda_{\tau}$.
Then the analytic continuation of $\Psi_p(z) $ along $\ell_{j}$ satisfies%
\begin{equation}
\ell_{j}^{\ast}\Psi_{p}(z)=\Psi_{p}(z),\text{ }j=1,2.   \label{304-1}
\end{equation}
\end{lemma}

\begin{proof}
We only need to prove (\ref{304-1}) in a small neighborhood $U$ of the base
point $q_{0}$. Since $\ell_{j}\in \pi_{1}\left( E_{\tau}\backslash \left \{
\pm \left[ p\right] \right \} ,q_{0}\right)$ does not intersect with $%
\ell_{p}+\Lambda_{\tau}$, $\Psi_{p}(z)$ can be viewed as a single-valued
meromorphic function in $\mathbb{C}\backslash(\ell_{p}+\Lambda_{\tau})$, and
in this region we have%
\begin{equation*}
\ell_{j}^{\ast}\Psi_{p}(z)=\Psi_{p}( z+\omega_{j}) =\pm \Psi _{p}(z),\text{ }%
z\in U\text{,}   \label{304}
\end{equation*}
because $\Psi_{p}(z)^{2}=\frac{\sigma(z)^{2}}{\sigma ( z+p) \sigma ( z-p) }$
is an elliptic function. Suppose
\begin{equation}
\Psi_{p}( z+\omega_{j})=-\Psi_{p}(z),\text{ }z\in U   \label{304-2}
\end{equation}
holds true. By fixing $q_{0}$ and $\ell_{j}$, (\ref{304-2}) always holds
true as $p\rightarrow0$ along $\ell_{p}$. Note from (\ref{61-3}) that for
any $z\in\ell_{j}$, $\lim_{p\rightarrow0}\Psi _{p}(z)$ is identical to
either $1$ or $-1$, but \eqref{304-2} implies that $\lim_{p\rightarrow0}\Psi
_{p}(z)$ for $z$ along $\ell_j$ contains both $1$ and $-1$, a contradiction.
\end{proof}

Classically, it has been known (cf. \cite{Whittaker-Watson}) that the second
symmetric product equation for any second order ODE play an important role.
Let $\tilde{y}_1(z)$, $\tilde{y}_2(z)$ are any two solutions of GLE$(\mathbf{%
n},p,A,\tau)$ and set $\Phi(z)=\tilde{y}_1(z)\tilde{y}_2(z).$ Then $\Phi(z)$
satisfies the following third order ODE:
\begin{equation}
\Phi^{\prime \prime \prime}(z)-4I_{\mathbf{n}}(z;p,A,\tau)\Phi^{%
\prime}(z)-2I_{\mathbf{n}}^{\prime}(z;p,A,\tau)\Phi(z)=0.   \label{303-1}
\end{equation}

Recall (\ref{304-111})-\eqref{304-12} that $y_1(z)$ is an common
eigenfunction and $y_2(z)=y_1(-z)$. Then $\Phi(z):=y_1(z)y_2(z)$ is a
solution of \eqref{303-1}, and also an \emph{even elliptic function} due to %
\eqref{304-111}-\eqref{304-12}. The following result was proved by
Takemura \cite{Takemura}, but we give a proof here for the convenience of
readers because it plays a fundamental role in our theory.

\begin{proposition}
\label{propp1} \cite{Takemura} The dimension of the space of even elliptic
solutions to (\ref{303-1}) is $1$.
\end{proposition}

\begin{proof}
It is easy to see that the dimension of the space of even solutions to %
\eqref{303-1} is $2$. Suppose the proposition is not true. Since we already
know there is at least one even elliptic solution, the dimension of even
elliptic solutions to \eqref{303-1} is $2$, which implies any even solution
of \eqref{303-1} must be elliptic.

Since the local exponents of GLE$(\mathbf{n},p,A,\tau)$ at $0$ are $-n_0,
n_0+1$, and $I_{\mathbf{n}}(\cdot;p,A,\tau)$ is even, there are local
solutions of the following form at $0$:
\begin{equation*}
\hat{y}_{1}(z)=z^{-n_{0}}\bigg(1+\sum_{j=1}^{\infty}a_{j}z^{2j}\bigg)  ,%
\text{ \ }\hat{y}_{2}(z) =z^{n_{0}+1 }\bigg(1+\sum_{j=1}^{\infty}b_{j}z^{2j}%
\bigg)  .
\end{equation*}
Then $\hat{y}_j(z)^2$, $j=1,2$, are even solutions of \eqref{303-1} and
hence even elliptic functions by our assumption. Define
\begin{equation}  \label{2-15}
\tilde{y}_j(z):=\frac{\hat{y}_j(z)}{\Psi _{p}(z)},\quad j=1,2,
\end{equation}
where $\Psi_{p}(z)$ is given by (\ref{61-3}). Then $\tilde{y}_j(z)$ is a
meromorphic function with poles at most at $E_{\tau}[2]$. Since $\hat{y}%
_j(z)^2$ and $\Psi_{p}(z)^2$ are even elliptic, so do $\tilde{y}_j(z)^2$ for
$j=1,2$. Assume
\begin{equation*}
\begin{pmatrix}
\tilde{y}_{1}( z+\omega_{i}) \\
\tilde{y}_{2}( z+\omega_{i})%
\end{pmatrix}
=
\begin{pmatrix}
a_i & b_i \\
c_i & d_i%
\end{pmatrix}
\begin{pmatrix}
\tilde{y}_{1}( z) \\
\tilde{y}_{2}( z)%
\end{pmatrix}%
,\quad i=1,2
\end{equation*}
for some $%
\begin{pmatrix}
a_i & b_i \\
c_i & d_i%
\end{pmatrix}%
\in SL(2,\mathbb{C})$. Then%
\begin{align*}
\tilde{y}_{1}(z) ^{2} =\tilde{y}_{1}( z+\omega _{i})^{2}=a_i^{2}\tilde{y}%
_{1}( z) ^{2}+2a_ib_i\tilde{y}_{1}(z) \tilde{y}_{2}(z) +b_i^{2}\tilde{y}%
_{2}(z)^{2}.
\end{align*}
So we have $a_i^{2}=1$, $b_i=0$. Together with $a_id_i-b_ic_i=1$, we see
that $d_i=a_i=:\varepsilon_{i}\in \{ \pm1\}$. Similarly, we could use $%
\tilde{y}_2$ to obtain $c_i=0.$ Hence
\begin{equation*}
\tilde{y}_{k}(z+\omega_{i}) =\varepsilon_{i}\tilde{y}_{k}( z) \text{ for }%
k=1,2\text{ and }i=1,2.
\end{equation*}
This implies that $\tilde{y}_{1}(z)\tilde{y}_{2}( z)$ is odd and elliptic.
By the local exponent of $\tilde{y}_{1}(z) \tilde{y}_{2}( z)$ at $\frac{%
\omega_{k}}{2}$ being one of $\{-2n_k,1, 2n_k+2\}$ for $k=1,2,3$, we see
that $\frac{\omega_{k}}{2}$ is a simple zero of $\tilde{y}_{1}(z) \tilde{y}%
_{2}(z)$ for $k=1,2,3$. Thus the elliptic function $\tilde{y}_{1}(z)\tilde{y}%
_{2}(z)$ has only a simple pole at $z=0$, a contradiction. This proves the
proposition.
\end{proof}

Now we can apply Proposition \ref{propp1} to answer the question proposed
earlier.

\begin{proposition}
\label{propp2} Let $y_1(z)$ be a common eigenfunction of the monodromy
representation of GLE$(\mathbf{n},p,A,\tau)$. Then GLE$(\mathbf{n},p,A,\tau)$
is completely reducible if and only if $y_1(z)$ and $y_2(z):=y_1(-z)$ are
linearly independent.
\end{proposition}

\begin{proof}
Clearly $y_2(z)$ is also a common eigenfunction. So the sufficient part is
trivial. For the necessary part, since the monodromy is completely
reducible, there exists another common eigenfunction $y_3(z)$ which is
linearly independent with $y_1(z)$. Then $y_3(z)y_3(-z)$ is also an even
elliptic solution of \eqref{303-1}. So Proposition \ref{propp1} implies $%
y_1(z)y_1(-z)=y_3(z)y_3(-z)$ up to a constant, which implies $y_3(z)=cy_1(-z)
$ for some constant $c\neq 0$. This proves that $y_1(z)$ and $y_2(z)=y_1(-z)$
are linearly independent.
\end{proof}

Recall \eqref{304-111} and \eqref{304-12} that $\varepsilon_j,
\varepsilon_j^{-1}$ are the eigenvalues of $M_j=\rho_{\tau}(\ell_j)$. Then
Propositions \ref{propp1} and \ref{propp2} have the following consequence.

\begin{corollary}
\label{coro11}If GLE$(\mathbf{n},p,A,\tau)$ is completely reducible then $%
(\varepsilon_1,\varepsilon_2)\notin\{\pm(1,1),\pm(1,-1)\}$.
\end{corollary}

Suppose GLE$(\mathbf{n},p,A,\tau)$ is completely reducible and $y_j(z),j=1,2$%
, are its linearly independent common eigenfunctions in Proposition \ref%
{propp2}. then the monodromy matrice $M_j$ can be written as
\begin{equation}  \label{eq216}
M_1=\left(%
\begin{matrix}
e^{-2\pi is} & 0 \\
0 & e^{2\pi is}%
\end{matrix}%
\right)\text{ and }M_2=\left(%
\begin{matrix}
e^{2\pi ir} & 0 \\
0 & e^{-2\pi ir}%
\end{matrix}%
\right),
\end{equation}
where $(r,s)\in\mathbb{C}^2$ satisfies $(\varepsilon_1,\varepsilon_2)=(e^{-2\pi is},e^{2\pi i r})$. We call that $(r,s)$ are the monodromy data of
GLE$(\mathbf{n},p,A,\tau)$ if it is completely reducible, and Corollary \ref%
{coro11} implies
\begin{equation}  \label{217}
(r,s)\notin\tfrac{1}{2}\mathbb{Z}^2.
\end{equation}
Later we will show how to compute $(r,s)$ from the zero set of the common
eigenfunction $y_1(z)$.

The above argument shows that \eqref{303-1} has a unique even elliptic
solution $\Phi(z)$ up to a constant, which is given by $\Phi(z)=y_1(z)y_2(z)$%
, where $y_1(z)$ is a common eigenfunction and $y_2(z)=y_1(-z)$.
Furthermore, GLE$(\mathbf{n},p,A,\tau)$ is completely reducible if and only
if $y_1(z)$ and $y_2(z)$ are linearly independent. This is where the even
elliptic solution $\Phi$ plays the role which is of fundamental importance
for GLE$(\mathbf{n},p,A,\tau)$. To determine the linear independence, we
should consider the Wronskian of $y_1(z)$ and $y_2(z)$: $W:=y_1(z)y_2^{%
\prime }(z)-y_1^{\prime }(z)y_2(z)$, which is a constant independent of $z$.

\begin{lemma}
\label{lemmm} Let $\Phi(z)=y_1(z)y_2(z)$ be the even elliptic solution of %
\eqref{303-1} and $W$ is the Wronkian of $y_1$ and $y_2$. Then
\begin{equation}  \label{eq218}
W^2=\Phi'(z)^{ 2}-2\Phi^{\prime \prime }(z)\Phi(z)+4I_{\mathbf{n}}(z;
p,A,\tau)\Phi(z)^2,
\end{equation}
\begin{equation}  \label{eq219}
y_1(z)=\sqrt{\Phi(z)}\exp\int^z\frac{W}{2\Phi(\xi)}d\xi
\end{equation}
\end{lemma}

\begin{proof}
Since $\Phi(z)=y_1(z)y_2(z)$ and $W=y_1^{\prime }y_2-y_1y_2^{\prime }$, we
have
\begin{equation*}
\frac{y_1^{\prime }}{y_1}=\frac{\Phi^{\prime }+W}{2\Phi},\quad\frac{%
y_2^{\prime }}{y_2}=\frac{\Phi^{\prime }-W}{2\Phi}
\end{equation*}
which implies \eqref{eq219} and
\begin{equation*}
\frac{\Phi^{\prime \prime }}{2\Phi}-\frac{\Phi^{\prime }+W}{2\Phi^2}%
\Phi^{\prime }=\left(\frac{y_1^{\prime }}{y_1}\right)^{\prime }=\frac{%
y_1^{\prime \prime }}{y_1}-\left(\frac{y_1^{\prime }}{y_1}\right)^2=I_{\mathbf{n}}-\left(%
\frac{\Phi^{\prime }+W}{2\Phi}\right)^2,
\end{equation*}
\begin{equation*}
\frac{\Phi^{\prime \prime }}{2\Phi}-\frac{\Phi^{\prime }-W}{2\Phi^2}%
\Phi^{\prime }=I_{\mathbf{n}}-\left(\frac{\Phi^{\prime }-W}{2\Phi}\right)^2.
\end{equation*}
Adding these two formulas together, we easily obtain (\ref{eq218}).
\end{proof}

To normalize the even elliptic solution $\Phi(z)$, we apply the following result due
to Takemura \cite{Takemura}.\medskip

\noindent \textbf{Theorem 2.A.} (\cite{Takemura}) \emph{Fix $\tau \in
\mathbb{H}$ and $p\not \in E_{\tau}[2]$. Then equation (\ref{303-1}) has a
unique even elliptic solution $\Phi_{e}(z;A)$ of the form
\begin{equation}
\Phi_{e}(z;A)=C_{0}(A)
+\sum_{k=0}^{3}\sum_{j=0}^{n_{k}-1}b_{j}^{(k)}(A)\wp(z+\tfrac{\omega_{k}}{2}%
)^{n_{k}-j}+\frac{d(A)}{\wp(z)-\wp (p)},  \label{5}
\end{equation}
such that the coefficients $C_{0}(A)$, $b_{j}^{(k)}(A)$ and $d(A)$ are all
polynomials in $A,$ and they do not have common zeros, and the leading
coefficient of $C_{0}(A)$ is $\frac{1}{2}$. Moreover, }%
\begin{equation*}
\deg_{A}C_{0}(A)>\max \left( \deg_{A}b_{j}^{(k)}(A),\deg_{A}d(A)\right) .
\end{equation*}

The proof of Theorem 2.A is not difficult but a little tedious in
computation. The expression of \eqref{5} is due to the fact that $\frac{%
\omega_k}{2}$ and $\pm p$ might be poles of $\Phi(z)$ with order $-2n_k$ and
$-1$ respectively (depending on $A$). We substitute \eqref{5} into GLE$(%
\mathbf{n},p,A,\tau)$ and compare $A$ of the two sides at each singularity.
Then we could obtain that all the coefficients are polynomials in $A$ after normalization. For
details, we refer the readers to \cite{Takemura,Takemura1}.

Recalling Proposition \ref{propp1}, we apply Lemma \ref{lemmm} to the normalized $\Phi_e(z;A)$ of Theorem
2.A. Then we have the main result of this section.

\begin{theorem}\label{thm2-7}
There is a monic polynomial $Q_{\mathbf{n,}p}(A)=Q_{\mathbf{n,}p}(A;\tau)$
in $A$ such that the Wronskian of $y_1(z;A)$, $y_2(z;A)=y_1(-z;A)$, where $%
\Phi_e(z;A)=y_1(z;A)y_2(z;A)$, satisfies
\begin{equation}
W^2=Q_{\mathbf{n,}p}(A).
\end{equation}
Moreover, $Q_{\mathbf{n,}p}(A)\neq 0$ if and only if GLE$(\mathbf{n}%
,p,A,\tau)$ is completely reducible.
\end{theorem}

\begin{proof}
By (\ref{5031}) and (\ref{101}), we could write $I(z;\mathbf{n},p,A,\tau)$
as
\begin{equation*}
I_{\mathbf{n}}(z;p,A,\tau) =A^{2}+I_{1}(z) A+I_{2}(z).
\end{equation*}
By Theorem 2.A, we have%
\begin{equation*}
\Phi_{e}(z)=\frac{1}{2}A^g+\sum_{j=0}^{g-1}\varphi_{j}(z)A^{j},\quad \text{%
where}\;\;g:=\deg C_0(A),
\end{equation*}
Inserting these into \eqref{eq218}, we easily obtain
\begin{align}
W^{2} =A^{2g+2}+\sum_{j=0}^{2g+1}q_{j}(z) A^{j},
\end{align}
where $q_{j}(z)\equiv c_{j}$ are independent of $z$ for all $j$ because $W$
is independent of $z$. Thus, $W^2=Q_{\mathbf{n},p}(A)$ is a monic polynomial
in $A$ of degree $2g+2$.
\end{proof}

In view of Theorem \ref{thm2-7}, we define the hyperelliptic curve $\Gamma_{\mathbf{n,%
}p}=\Gamma_{\mathbf{n,}p}(\tau)$ by
\begin{equation}
\Gamma_{\mathbf{n,}p}(\tau):=\{(A,W)|W^2=Q_{\mathbf{n,}p}(A;\tau)\}.
\end{equation}
We remark that the polynomials $Q_{\mathbf{n,}p}(A;\tau)$ might have
multiple zeros. But we could prove that for each $\tau$, $Q_{\mathbf{n,}%
p}(A;\tau)$ has distinct roots except for finitely many $p\in
E_\tau\setminus E_\tau[2]$. See e.g. our subsequent work \cite{CKL-pre}.

Since $\deg_{A}Q_{\mathbf{n,}p}(A;\tau)  $ is even, the curve
$\Gamma_{\mathbf{n,}p}(\tau)  $
has two points at infinity which are denoted by $\infty_{\pm}$. Indeed, for a curve in
$\mathbb{C}^{2}$ defined by $y^{2}=\Pi_{i=1}^{2g+2}(x-x_{i})  $,
to study its points at infinity, we let
$x=1/x^{\prime}$ and $y=y^{\prime}/x^{\prime g+1}$. Then the equation becomes
$
y^{\prime2}=\prod_{i=1}^{2g+2}(  1-x^{\prime}x_{i})  .
$
Thus $(  x^{\prime},y^{\prime})=(  0,\pm1)  $
represents the two points at infinity and they are unramified. Hence
\begin{equation}\label{eqgamma}\overline{\Gamma_{\mathbf{n,}%
p}(  \tau )  }=\Gamma_{\mathbf{n,}p}(  \tau )
\cup \{  \infty_{\pm} \}  \;\,\text{is smooth at}\;\,\infty_{\pm}.\end{equation}

\begin{remark}\label{remarkkk}
Given $(A,W)\in\Gamma_{n,p}$, the unique $\Phi_e(z;A)$ is the product of $%
y_1(z;A)$ and $y_2(z;A)=y_1(-z;A)$. If $W\neq 0$, then $y_1(z;A)$ and $%
y_2(z;A)$ are linearly independent. We rename $y_1(z;A), y_2(z;A)$ by
requiring that the Wronskian of $y_1$ and $y_2$ is equal to $W$ (i.e. the
Wronskian of $y_2$ and $y_1$ is $-W$). Then $y_1(z;A)$ is unique up to a
sign, and will be denoted by $y_1(z;A,W)$. In particular, the zero set of $%
y_1(z;A,W)$ is unique.
\end{remark}

\section{The embedding of $\Gamma_{\mathbf{n,}p}$ in ${Sym}^NE_{\protect\tau}
$}

\label{embedding}

The main purpose of this section is to define the addition map from $\Gamma_{%
{\mathbf{n}, p}}$ to $E_\tau$. First we discuss the embedding of $\Gamma_{%
\mathbf{n,}p}$ into Sym$^NE_\tau$, the symmetric $N$-th copy product of $%
E_\tau$, where $N=\sum_{k=0}^{3}n_{k}+1$ in the sequel.

For any $\mathbf{a}=( a_{1},\cdot \cdot \cdot ,a_{N})\in
\mathbb{C}^{N}$, we define $y_{{\bf a},c}( z)=y_{{\bf a},c}( z;p)$ by
\begin{equation}
y_{{\bf a},c}( z;p) :=\frac{e^{cz}\prod_{i=1}^{N}\sigma (
z-a_{i}) }{\sqrt{\sigma ( z-p) \sigma ( z+p) }%
\prod_{k=0}^{3}\sigma (z-\frac{\omega _{k}}{2})^{n_{k}}}\text{ for }%
c\in \mathbb{C}.  \label{exp1}
\end{equation}

\begin{proposition}\label{propp3}
Fix $p \in E_{\tau }\setminus E_{\tau}[2]$ and $A\in \mathbb{C}.$ Let $y_1(z;A,W)$ be the common eigenfunction determined in Remark \ref{remarkkk}. Then there
always exists ${\mathbf a}\in \mathbb{C}^{N}$ and $c\in \mathbb{C}$ such that $%
y_1(z;A,W)=y_{{\mathbf a},c}(z;p)$ up to a constant.
\end{proposition}

\begin{proof}
Since $y_1(z;A,W)$ is an common eigenfunction, we have
\begin{equation}\label{exp11}
\ell_i^*y_1(z;A,W)=\varepsilon_{i}y_1(z;A,W)\text{ for some }\varepsilon_{i}\neq 0, i=1,2.
\end{equation}%
As we discussed in Section \ref{monodromy}, $y_1(z;A,W)$ has branch points at $\pm p$. Set $%
\tilde{y}(z):=y_1(z;A,W)/\Psi_p(z)$, where $\Psi_p(z)$ is defined in (\ref{61-3}).
Then $\tilde{y}(z)$ is a
meromorphic function, and it follows from Lemma \ref{lem2-6} and \eqref{exp11} that
\begin{equation}\label{exp111}
\tilde{y}(z+\omega_i)=\varepsilon_{i}\tilde{y}(z),\quad i=1,2.
\end{equation}%
Conventionally, a meromorphic function satisfying %
\eqref{exp111} is called \emph{an elliptic function of second kind} with periods $1$ and $\tau$. Then a classic theorem says that up to a constant, $\tilde{y}(z)$ can be written as
\begin{equation}\label{ytilde}
\tilde{y}(z)=\frac{e^{cz}\prod_{i=1}^{N}\sigma (z-a_{i})}{\sigma
(z)\prod_{k=0}^{3}\sigma (z-\frac{\omega _{k}}{2})^{n_{k}}},
\end{equation}
because $\tilde{y}(z)$ have poles at most at $0$ with order $-n_0-1$ and at $\omega _{k}/2$ with order $-n_k$, $k=1,2,3$. The proof is complete.
\end{proof}

\begin{remark}\label{remarkk1} A consequence of Remark \ref{remarkkk} and Proposition \ref{propp3} is that, up to a constant, the unique even elliptic solution $\Phi_{e}(z;A)$ in Theorem 2.A can be written as
\begin{equation}\label{phi-a}
\Phi_{e}(z;A)=\frac{\prod_{i=1}^{N}\sigma (
z-a_{i})\sigma (
z+a_{i}) }{\sigma ( z-p) \sigma ( z+p)
\prod_{k=0}^{3}\sigma (z-\frac{\omega _{k}}{2})^{n_{k}}\sigma (z+\frac{\omega _{k}}{2})^{n_{k}}}.
\end{equation}
On the other hand, if $y_{\mathbf{a},c}(z)$ is a solution of GLE$({\bf n},p,A,\tau)$, it is easy to see that (i) $y_{-\mathbf{a},-c}(z)=d_1y_{\mathbf{a},c}(-z)$ (for some constant $d_1\neq 0$) is also a solution of GLE$({\bf n},p,A,\tau)$, (ii) $y_{{\bf a},c}(z)$ is a common eigenfunction of $M_j$'s. Together with Proposition \ref{propp3}, we have that (iii) $y(z)$ is a common eigenfunction if and only if  $y(z)$ is of the form $y_{\mathbf{a},c}(z)$ up to a constant, and (iv) $y_2(z):=y_1(-z;A,W)=y_{-{\mathbf a},-c}(z;p)$ up to a constant. In particular, Proposition \ref{propp2} implies that GLE$(\mathbf{n},p,A,\tau)$
is completely reducible if and only if $y_{\mathbf{a},c}(z)$ and $y_{-\mathbf{a},-c}(z)=d_1y_{\mathbf{a},c}(-z)$ are
linearly independent.
\end{remark}

Proposition \ref{propp3} says that $y_1(z;A,W)=dy_{\mathbf{a},c}(z)$ for some $%
\mathbf{a}=(a_1,\cdots,a_N)$ $\in\mathbb{C}^N$, $c\in\mathbb{C}$ and $d\neq 0$. Obviously $%
\{a_1,\cdots,a_N\}\setminus (E_{\tau}[2]\cup\{\pm [p]\})$ is the zero set of $y_1(z;A,W)$ and we expect that $c$ can be
uniquely determined by $\mathbf{a}$. On the other hand, we could apply the
transformation law \eqref{123123} of $\sigma$ to obtain the monodromy
data $(r,s)$ in \eqref{eq216} if GLE$({\bf n},p,A,\tau)$ is completely reducible. These
are proved in the following result. Denote $\eta_3:=\eta_1+\eta_2$.

\begin{theorem}
\label{thm6.1}Let $p\notin E_{\tau}[2]$ and $A\in \mathbb{C}$. Suppose GLE$({\bf n},p,A,\tau)$ is completely reducible with solution $y_{\mathbf{a},c}(z)$ given in Proposition \ref{propp3}. Then the monodromy data $(r,s)$ in \eqref{eq216} and $c=c({\bf a})$ can be determined by $\mathbf{a}$ as follows:
\begin{equation}
\sum_{i=1}^{N}a_{i}=r+s\tau+\sum_{k=1}^{3}\frac{n_{k}\omega_{k}}{2}%
,   \label{61-37}
\end{equation}%
\begin{equation}
c({\bf a})=r\eta_{1}+s\eta_{2}=\frac{1}{2}\sum_{i=1}^{N}(\zeta
(a_{i}+p)+\zeta(a_{i}-p))-\sum_{k=1}^{3}\frac{n_{k}\eta_{k}}{2}, \label{61-38}
\end{equation}%
\begin{equation}
c({\bf a})=\frac{1}{2}\sum_{i=1}^{N}(\zeta
(a_{i}+\tfrac{\omega_k}{2})+\zeta(a_{i}-\tfrac{\omega_k}{2}))-\sum_{i=1}^{3}\frac{n_{i}\eta
_{i}}{2}\text{ if }n_k\not =0.  \label{kkkl}
\end{equation}
Furthermore, $a_i\notin E_\tau[2]$ for all $i$ and $\mathbf{a}=\{a_{1},\cdot \cdot
\cdot,a_{N}\}$ must satisfy one of the following three alternatives:

\begin{itemize}
\item[(a-i)] $a_{i}\not =\pm p$ and $a_{i}\not =\pm a_{j}$ for any $i\not =j$%
, i.e. $\pm p$ are simple poles of $\Phi_{e}(z;A)$.

\item[(a-ii)] $a_{N-1}=a_{N}=p$, $a_{i}\not =\pm p$ and $a_{i}\not =\pm a_{j}
$ for any $i\not =j\leq N-2$, i.e. $\pm p$ are simple zeros of $\Phi_{e}(z;A)$.

\item[(a-iii)] $a_{N-1}=a_{N}=-p$, $a_{i}\not =\pm p$ and $a_{i}\not =\pm
a_{j}$ for any $i\not =j\leq N-2$, i.e. $\pm p$ are simple zeros of $\Phi
_{e}(z;A)$.
\end{itemize}
\end{theorem}

\begin{proof}
Since GLE$({\bf n},p,A,\tau)$ is completely reducible, Remark \ref{remarkk1} says that $y_{\mathbf{a},c}(z)$ and $y_{\mathbf{a},c}(-z)$ are
linearly independent, which implies
\begin{equation*}
\begin{array}{l}
y_{\mathbf{a},c}(z)\text{ \textit{and} }y_{\mathbf{a},c}(-z) \text{ \textit{has a pole
of order} }-n_{k}\text{ \textit{at} }z=\frac{\omega_{k}}{2}\text{ \textit{if}
}n_{k}>0 \\
\text{\textit{and non-zero at} }z=\frac{\omega_{k}}{2}\text{ \textit{if} }%
n_{k}=0.%
\end{array}
\end{equation*}
Thus $a_i\notin E_{\tau}[2]$ for all $i$.

Recalling $\tilde{y}(z)$ in (\ref{ytilde}) and $\Psi_{p}(z)$ in (\ref{61-3}), we have
\begin{equation}
y_{\mathbf{a},c}(z)=\tilde{y}(z)\Psi_{p}(z).   \label{y-}
\end{equation}
By applying the
transformation law \eqref{123123} of $\sigma$ to $\tilde{y}(z)$, we can determine $\varepsilon_i$ in (\ref{exp11})-(\ref{exp111}) by
\[\varepsilon_i=\exp\bigg(c\omega_i-\eta_i\bigg(\sum_{i=1}^{N}a_{i}-\sum_{k=1}^{3}\frac{n_{k}\omega_{k}}{2}\bigg)\bigg),\quad i=1,2.\]
Define $(r,s)\in \mathbb{C}^2$ by
\begin{equation}\label{61-37-1}-2\pi i s:=c-\eta_1\bigg(\sum_{i=1}^{N}a_{i}-\sum_{k=1}^{3}\frac{n_{k}\omega_{k}}{2}\bigg)\end{equation}
\begin{equation}\label{61-37-2}2\pi i r:=c\tau-\eta_2\bigg(\sum_{i=1}^{N}a_{i}-\sum_{k=1}^{3}\frac{n_{k}\omega_{k}}{2}\bigg).\end{equation}
Then $(\varepsilon_1, \varepsilon_2)=(e^{-2\pi is},e^{2\pi ir})$, and it follows from (\ref{exp11}) that this $(r,s)$ is precisely the monodromy data in \eqref{eq216}. By the Legendre relation $\tau\eta_1-\eta_2=2\pi i$, we see that (\ref{61-37-1})-(\ref{61-37-2}) are equivalent to (\ref{61-37}) and $c=r\eta_1+s\eta_2$.

On the other hand, the second equality of (\ref{61-38}) can by proved by inserting the expression
(\ref{exp1}) of $y_{\bf{a}, c}(z)$ into GLE$({\bf n},p,A,\tau)$ and computing the leading terms at
singularities $\pm p$. If $n_{k}>0$, then $\frac{\omega_k}{2}$ is also
a singularity. Then (\ref{kkkl}) follows by inserting (\ref{exp1}) into GLE$({\bf n},p,A,\tau)$ and computing the leading terms at singularities $\pm p$ and $\frac{\omega_k}{2}$. We omit the details here.

To prove (i)-(iii), we recall Remark \ref{remarkkk} that $\Phi_e(z;A)=y_{\mathbf{a},c}(z)y_{\mathbf{a},c}(-z)$. Since $\Phi_e(z;A)$ is even, it has the same local exponent $\alpha$ at $p$ and $-p$, and $\alpha\in\{-1,1,3\}$.

If $\alpha=3$, then the leading term of both $y_{\mathbf{a},c}(z)$
and $y_{\mathbf{a},c}(-z)$ near $p$ is $( z-p) ^{\frac{3}{2}}$,
which implies that $y_{\mathbf{a},c}(z)$
and $y_{\mathbf{a},c}(-z)$ are
linearly dependent, a contradiction.

If $\alpha=-1$, i.e. $\pm p$ are both simple poles of $\Phi_{e}(z;A)$, then by (\ref{phi-a}) we have $a_{j}\not =\pm p$ in $E_\tau$
for all $j$. Consequently, $a_1,\cdots, a_N$ are all zeros of $y_{\mathbf{a},c}(z)$ and so are all simple, i.e. $a_i\neq a_j$ for $i\neq j$. Since $y_{\mathbf{a},c}(z)$
and $y_{\mathbf{a},c}(-z)$ can not have common zeros, we
also have $a_{i}\not =- a_{j}$ for all $i,j$. This proves (a-i).

If $\alpha=1$, i.e. $\pm p$ are both simple zeros of $\Phi_{e}(z;A)$,
one possibility is that the local exponent of $y_{\mathbf{a},c}(z)$ at $p$ is $\frac{3}{2}$ and at $-p$ is $%
\frac{-1}{2}$. Then (\ref{exp1}) implies that $-p\not \in\{a_{1},\cdot
\cdot \cdot,a_{N}\}$ and there are exactly two elements in $\left \{
a_{1},\cdot \cdot \cdot,a_{N}\right \} $ equal to $p$ in $E_{\tau}$. By reordering $%
a_{1},\cdot \cdot \cdot,a_{N}$, (a-ii) is proved. Similarly, for the other possibility that the
local exponent of $y_{\mathbf{a},c}(z)$ at $p$ is $\frac{-1}{2}$ and at $-p$ is $\frac{3%
}{2}$, (a-iii) holds.

The proof is complete.
\end{proof}

By the Legendre relation, the matrix $%
\begin{pmatrix}
1, & \tau \\
\eta_1(\tau) & \eta_2(\tau)%
\end{pmatrix}%
$ is always invertible. Hence \eqref{61-37}-\eqref{61-38} imply the
monodromy data $(r,s)$ can be uniquely determined by $a_1,\cdots,a_N$, which contain all the
zeros of $y_1(z;A,W)$. The next result is to show that $A$ of GLE$({\bf n}, p,A,\tau)$ can
be expressed in terms of $\{a_1,\cdots,a_N\}$ if Case (a-i) happens, i.e. $\pm p\notin\{a_1,\cdots,a_N\}$.

\begin{proposition}
\label{thmpole}Let $p\notin E_{\tau}[2]$ and $A\in \mathbb{C}$. Suppose GLE$({\bf n}, p,A,\tau)$ is completely reducible and Case (a-i) in
Theorem \ref{thm6.1} occurs, then
{\allowdisplaybreaks
\begin{align}
A= & \frac{1}{2}\sum_{i=1}^{N}( \zeta(a_{i}+p)-\zeta(a_{i}-p)) -%
\frac{1}{2}\zeta(2p)  \label{AAA} \\
& -\frac{1}{2}\sum_{k=0}^{3}n_{k}\left( \zeta(p+\tfrac{\omega_{k}}{2}%
)+\zeta(p-\tfrac{\omega_{k}}{2})\right).  \notag
\end{align}
}%
\end{proposition}

Again, (\ref{AAA}) can be obtained directly by inserting the expression
(\ref{exp1}) of $y_{\bf{a}, c}(z)$ into GLE$({\bf n},p,A,\tau)$ and computing the leading terms at
singularities $\pm p$. We omit the
proof here.

\begin{remark}
Theorem 2.A implies that for fixed $\tau$ and $p\notin E_{\tau}[2]$, there are \emph{only finite many} $A$'s (i.e. zeros of the polynomial $d(A)$) such that Case (a-ii) or (a-iii) occurs.
Therefore, if $\mathbf{a}$ is in Case (a-ii) or (a-iii), we could use
a sequence of $\mathbf{a}^{k}$ of Case (a-i) to approximate $\mathbf{a}$. In
particular, $a_{N-1}^{k}$ and $a_{N}^{k}$ converge to $p$ for Case (a-ii).
The fact that $c(\mathbf{a}^{k}) $ converges and (\ref{61-38}) implies the sum $
\zeta( a_{N-1}^{k}-p) +\zeta( a_{N}^{k}-p) $ also
converges, namely $( a_{N-1}^{k}-p) ^{-1}+(
a_{N}^{k}-p) ^{-1}$ tends to a finite limit as $\mathbf{a}^{k}\to \mathbf{a}$.
Then $c( \mathbf{a}) $ and $A$ can be also expressed in terms of $\{a_{1},\cdot \cdot \cdot,a_{N-2}\}$ in Case (a-ii) and (a-iii).\end{remark}

Now we consider that GLE$(\mathbf{n},p,A_{0},\tau)$
is \textit{not completely reducible}. Then $y_{{\bf a},c}(z;p)$ in Proposition \ref{propp3}
is the only common eigenfunction up to a constant.

\begin{theorem}
\label{thm3}Let $p\not \in E_{\tau}[2] $ and $A\in \mathbb{C}$.
Suppose that GLE$(\mathbf{n},p,A,\tau)$ is not completely
reducible with solution $y_{\mathbf{a},c}(z)$ given in Proposition \ref{propp3}. Then There exists $( m_{1},m_{2}) \in \frac{1}{2}
\mathbb{Z}^{2}$ such that
\begin{equation}
\{ a_{1},\cdot \cdot \cdot,a_{N}\} \equiv\{ -a_{1},\cdot \cdot
\cdot,-a_{N}\} \text{ mod }\Lambda_\tau,   \label{a}
\end{equation}%
\begin{equation}
\sum_{j=1}^{N}a_{j}=m_{1}+m_{2}\tau+\sum_{k=1}^{3}\frac{%
n_{k}\omega_{k}}{2}\in E_{\tau}[2] ,   \label{b}
\end{equation}%
\begin{equation}
c=m_{1}\eta_{1}+m_{2}\eta_{2}.   \label{c}
\end{equation}
\end{theorem}

\begin{proof}
Remark \ref{remarkk1} shows that $y_{{\bf a},c}(-z)$ and $y_{{\bf a},c}(z)$ are linearly dependent, so
\eqref{a} follows trivially from the expression (\ref{exp1}). Besides, (\ref{exp11}) gives
\[
\ell_i^*y_{{\bf a},c}(z)=\varepsilon_{i}y_{{\bf a},c}(z),\quad \varepsilon_{i}\in \{\pm 1\},\; i=1,2,
\]
because the linear dependence of $y_{{\bf a},c}(-z)$ and $y_{{\bf a},c}(z)$ imply $\varepsilon_{i}=\varepsilon_{i}^{-1}$. Consequently,
\eqref{b}-\eqref{c} can be proved by the same way as (\ref{61-37})-(\ref{61-38}), where $( m_{1},m_{2}) \in \frac{1}{2}
\mathbb{Z}^{2}$ follows from $\varepsilon_{i}\in \{\pm 1\}$ for $i=1,2$.
\end{proof}

Now we want to define the embedding of $\Gamma_{n,p}$ into Sym$^NE_\tau$,
where $N=\sum_{i=0}^3n_i+1$ and Sym$^NE_\tau$ is the $N$-th symmetric product of $E_\tau $. For any $(A,W)\in\Gamma_{n,p}$, it follows from Remark \ref{remarkkk} that
the solution $y_1(z;A,W)$ is unique up to a sign. By (\ref{exp1}) and Proposition \ref{propp3}, there is  ${\bf a}=\{a_1,\cdots,a_N\}$ (unique mod $\Lambda_{\tau}$) such that $y_1(z;A,W)=y_{{\bf a},c}(z)$, i.e. {\bf a} gives
all the zeros of $y_1(z;A,W)$. Then we define a map $i_{n,p}:\Gamma_{n,p}\rightarrow \,\text{Sym}^NE_\tau$
by
\begin{equation}
i_{\mathbf{n},p}(A,W) :=\{[a_{1}],\cdot \cdot \cdot,[a_{N}]\}\in \text{Sym}^NE_\tau,
\label{i1}
\end{equation}
where as introduced in Section 1, $[a_i]:=a_i \ (\text{mod}\ \Lambda_{\tau}) \in E_{\tau}$. The above argument shows that $i_{{\bf n},p}$ is well-defined. Furthermore, if $W\neq0$, then we see from Remark \ref{remarkkk} that
\begin{equation}
i_{\mathbf{n},p}( A,-W) =\{-[a_{1}],\cdot \cdot \cdot,-[a_{N}]\}.
\label{i2}
\end{equation}

\begin{proposition}
$i_{{\bf n},p}$ is an embedding from $\Gamma_{n,p}$ into Sym$^NE_\tau$.
\end{proposition}

\begin{proof}
Suppose $i_{{\bf n},p}(A,W)=i_{{\bf n},p}(\tilde{A},\tilde{W})=\{[a_1],\cdots,[a_N]\}$.
 Then (\ref{phi-a}) implies $
\Phi_e(z;A)=\Phi_e(z;\tilde{A})$ up to a constant and so $A=\tilde{A}$ by (\ref{303-1}). Together with Theorem
\ref{thm2-7}, we have $W^2=\tilde{W}^2$.
If $W=0$ then $\tilde{W}=0.$ If $W\neq 0$ then by \eqref{i1} and \eqref{i2}
we also have $W=\tilde{W}$. This proves the embedding of $i_{{\bf n},p}$.
\end{proof}

By applying $i_{{\bf n},p}$, we could define $\sigma_{{\bf n},p}:\Gamma_{n,p}\rightarrow
E_\tau$ by
\begin{equation}
\sigma_{{\bf n},p}(A,W):=\sum_{i=1}^N[a_i]-\sum_{k=1}^3 [\tfrac{n_k\omega_k}{2}],
\end{equation}
which is the composition of $i_{{\bf n},p}$ and the addition map $\{[a_1],\cdots,[a_N]\}\mapsto \sum_{i=1}^N[a_i]-\sum_{k=1}^3 [\tfrac{n_k\omega_k}{2}]$. Thus $i_{{\bf n},p}$ is also called the \emph{addition map}. Clearly
\[\sigma_{{\bf n},p}(A,-W)=-\sum_{i=1}^N[a_i]-\sum_{k=1}^3 [\tfrac{n_k\omega_k}{2}]=-\sigma_{{\bf n},p}(A,W).\]
The map $\sigma_{{\bf n},p}$ is a holomorphic map (or a finite morphism) from $\Gamma_{{\bf n},p}$
to $E_\tau$ if the irreducible curve $\Gamma_{{\bf n},p}$ is smooth (or if $\Gamma_{n,p}$ is
singular). Hence in both cases, the degree $\deg \sigma_{{\bf n},p}= \#\sigma_{{\bf n}, p}^{-1}(z),z\in E_\tau$,
is well-defined. How to calculate it is the main purpose of the paper.

\begin{definition} We let $Y_{{\bf n},p}(\tau)$ be the image of $\Gamma_{{\bf n},p}(\tau)$ in Sym$^NE_\tau$
under $i_{{\bf n},p}$, and $X_{{\bf n},p}(\tau)$ be the image of $\{(A,W)\in\Gamma_{{\bf n},p}|W\neq
0\}$ under $i_{{\bf n},p}$.
\end{definition}

The next two section will discuss of the limiting of $Y_{{\bf n},p}(\tau)$ for two case (i) $%
A\rightarrow\infty$ and $p\notin E_\tau[2]$ is fixed; (ii) $%
p\rightarrow\omega_k/2$.

\section{The closure of $Y_{\mathbf{n},p}\left( \protect\tau \right) $}

\label{closure}


The purpose of this section is to prove the following result.

\begin{theorem}
\label{Hyper}Let $\mathbf{n}=$ $\left( n_{0},n_{1},n_{2},n_{3}\right) $
where $n_{k}\in \mathbb{Z}_{\geq0}$. Then the followings hold.

\begin{itemize}
\item[(i)]
\begin{equation*}
\overline{X_{\mathbf{n},p}( \tau) }=\overline{Y_{\mathbf{n}%
,p}(\tau) }=Y_{\mathbf{n},p}( \tau ) \cup  \{
\infty_{+}(p),  \infty_{-}(p)\},
\end{equation*}
where
\begin{equation}
\infty_{\pm}( p) :=\bigg( \overset{n_{0}}{\overbrace{0,\cdot \cdot
\cdot,0}},\overset{n_{1}}{\overbrace{\frac{\omega_{1}}{2},\cdot \cdot \cdot,%
\frac{\omega_{1}}{2}}},\overset{n_{2}}{\overbrace{\frac{\omega_{2}}{2},\cdot
\cdot \cdot,\frac{\omega_{2}}{2}}},\overset{n_{3}}{\overbrace {\frac{%
\omega_{3}}{2},\cdot \cdot \cdot,\frac{\omega_{3}}{2}}},\pm p\bigg).   \label{+}
\end{equation}%
Furthermore, $\overline{X_{\mathbf{n},p}(\tau) }$
is smooth at $\infty_{\pm}(p) $.

\item[(ii)] The map $i_{\mathbf{n},p}:\Gamma_{\mathbf{n,}%
p}\rightarrow Y_{\mathbf{n},p}$ has a natural extension to $\bar{%
\imath}_{\mathbf{n},p}:\overline{\Gamma_{\mathbf{n,}p}}\rightarrow
\overline {Y_{\mathbf{n},p}}$ by $\bar{\imath}_{%
\mathbf{n,}p}( \infty_{\pm} ): =\infty_{\pm}(
p)  $, where $\{ \infty_{\pm}\}:=\overline{\Gamma_{\mathbf{n,}p}}\backslash \Gamma_{%
\mathbf{n},p}$ is given in (\ref{eqgamma}).
\end{itemize}

\end{theorem}

We will prove a slight generalization of Theorem \ref{Hyper}. Indeed, we
want to study the limiting of $Y_{\mathbf{n},p}(\tau) $ as $%
p\rightarrow p_{0}\not \in E_{\tau}[2] $. For each $p$ near $%
p_0$, we associate a $A(p)\in\mathbb{C}$ and consider GLE$(\mathbf{n},p,A(p),
\tau)$.
Letting ${\bf a}(p)=\{[a_1(p)],\cdots,[a_N(p)]\}=i_{{\bf n}, p}(A(p),W(p))$ and $c(p):=c({\bf a}(p))=c$ given by Proposition \ref{propp3}, we study the limits of ${\bf a}(p)$ and $c(p)$ as $p\rightarrow p_0$. Since ${\bf a}(p)\in \text{Sym}^N{E_{\tau}}$, up to a subsequence, we can always assume that ${\bf a}(p)\to {\bf a}^{0}=\{[a_1^0],\cdots,[a_N^0]\}$ and
\begin{equation}\label{ajconverge}
a_j(p)\to a_j^{0}\quad\text{as}\;\,p\to p_0,\quad \forall\, j.
\end{equation}

\begin{proposition}
\label{prop5-0} Let $p\rightarrow p_{0}\not \in E_{\tau}[2]$ and $
S:=E_\tau[2]\cup\{\pm [p_0]\}$. Then $I_{\mathbf{n}}( z;p,A(p),\tau)$ converges in $E_\tau\setminus S$ if and only if $A(p)$
converges if and only if the corresponding $B(p)$ converges if and only if $c(p)$ converges.
\end{proposition}

\begin{proof}
From \eqref{5031}, it is easy to see that $I_{\mathbf{n}}(z;p,A(p),\tau)$ converges for $z\in E_\tau\setminus S$ if and only if $A(p)$
converges if and only if $B(p)$ converges. To study its relation with $c(p)$, we
recall the solution $y_{{\bf a}(p),c(p)}(z;p)$ of GLE$(\mathbf{n},p,A(p),
\tau)$ given in Proposition \ref{propp3}. Then
\begin{align*}
\frac{y_{{\bf a}(p),c(p)}'(z;p)}{y_{{\bf a}(p),c(p)}(z;p)}
=&c(p) +\sum_{j=1}^{N}\zeta( z-a_{j}(p))
-\sum_{k=0}^{3}n_{k}\zeta ( z-\tfrac{\omega_{k}}{2}) \\
& -\tfrac{1}{2}( \zeta (z+p) +\zeta(z-p))
\end{align*}
and so
\begin{align}
&I_{\mathbf{n}}(z;p,A(p),\tau)=\frac{y_{{\bf a}(p),c(p)}''(z;p)}{y_{{\bf a}(p),c(p)}(z;p)} =\left( \frac{y_{{\bf a}(p),c(p)}'}{y_{{\bf a}(p),c(p)}}\right)
^{\prime}+\left( \frac{y_{{\bf a}(p),c(p)}'}{y_{{\bf a}(p),c(p)}}\right) ^{2}  \label{72} \\
&=c(p)^2 +2c(p) D( z;\mathbf{a}(p), p) +D( z;\mathbf{a}(p), p)^2+E( z;\mathbf{a}(p), p) ,  \notag
\end{align}
where%
\begin{equation}
D( z;\mathbf{a}(p),p) :=\left[
\begin{array}{c}
\sum_{j=1}^{N}\zeta( z-a_{j}(p))
-\sum_{k=0}^{3}n_{k}\zeta ( z-\tfrac{\omega_{k}}{2}) \\
-\tfrac{1}{2}( \zeta (z+p) +\zeta(z-p))%
\end{array}
\right]   \label{78}
\end{equation}
and $E( z;\mathbf{a}(p),p)=D'( z;\mathbf{a}(p),p)$ is given by
\begin{equation}
E( z;\mathbf{a}(p),p): =\left[
\begin{array}{c}
\sum_{k=0}^{3}n_{k}\wp( z-\frac{\omega_{k}}{2})
-\sum_{j=1}^{N}\wp( z-a_{j}(p)) \\
+\frac{1}{2}( \wp(z+p)+\wp(z-p))%
\end{array}
\right].   \label{79}
\end{equation}
By (\ref{ajconverge}), both $D( z;\mathbf{a}(p),p)$ and $E( z;\mathbf{a}(p),p)$ converge uniformly for $z\in E_{\tau}\backslash (\{[a_{j}^{0}]\}_{j=1}^{N}\cup\{\pm [p_0]\}\cup E_{\tau}[2]) $. Thus, $I_{\mathbf{n}}( z;p,A(p),\tau)$ converges if and only if $c(p)$ converges.
\end{proof}

Now we discuss the limit $\mathbf{a}^{0}$ of $\mathbf{a}(p) $
as $p\rightarrow p_{0}\not \in E_{\tau}[ 2] $. There are two
cases: (i) the corresponding $B(p) \rightarrow B_{0}\in\mathbb{C}$; (ii) $%
B(p)\to\infty$. Case (i) is easy, because $B(p)\rightarrow B_0\in\mathbb{C}$ implies $A(p)$ converges to some $A_0\in \mathbb{C}$, i.e. GLE$({\bf n},p,A(p),\tau)$ converges to GLE$({\bf n},p_0,A_0,\tau)$, and so $\lim_{p\to p_0}\mathbf{a}(p)=\mathbf{a}%
^0\in Y_{{\bf n},p_0}(\tau)$. We discuss case (ii)
in the following proposition.

\begin{proposition}
\label{prop5-3}Suppose $\mathbf{a}( p) \to \mathbf{a}%
^{0}$ and $B( p) \rightarrow \infty$ as $p\rightarrow p_{0}\not
\in E_{\tau}[ 2] $. Then $\mathbf{a}^0\in\{\infty_{\pm}(p_0)\}$, where $\infty_{\pm}(p)$ is defined in (\ref{+}).
\end{proposition}

\begin{proof}
Since $B(p)\to\infty$, by Proposition 4.2, we have both $A(p)$ and $%
c(p)\rightarrow\infty$ as $p\rightarrow p_0.$ We claim that $%
c(p)/A(p)\rightarrow\beta$ for some $\beta\in \mathbb{C}\setminus\{0\}$ as $p\to p_0$.
In fact,
by differentiating both expressions \eqref{5051} and \eqref{72} of $I_{\bf n}(z; p, A(p), \tau)$, we obtain
\begin{equation}  \label{qq}
2c(p)E(z;\mathbf{a}(p),p)=A(p)(\wp(z-p)-\wp(z+p))+\text{ other
terms,}
\end{equation}
where other terms remain bounded outside the singularities $\{[a_{j}^{0}]\}_{j=1}^{N}\cup\{\pm [p_0]\}\cup E_{\tau}[2]$ as $p\rightarrow p_0$. Therefore, $
c(p)/A(p)\to\beta\neq 0$. Since $\mathbf{a}^0=\lim_{p\rightarrow p_0}\mathbf{a%
}(p)$, \eqref{qq} yields $2\beta E(z;\mathbf{a}^0,p_0)=\wp(z-p_0)-\wp(z+p_0)$, and it follows from (\ref{79}) that
\begin{align}
& 2\beta \left[
\begin{array}{c}
\sum_{k=0}^{3}n_{k}\wp( z-\frac{\omega_{k}}{2})
-\sum_{j=1}^{N}\wp( z-a_{j}^0) \\
+\frac{1}{2}( \wp(z+p_0)+\wp(z-p_0))%
\end{array}
\right]  \label{D2} \\
& =\wp(z-p_0)-\wp(z+p_0) .  \notag
\end{align}
From (\ref{D2}), we have  two possibilities: One is
$\mathbf{a}^0 =\infty_{+}( p_{0})$ and $\beta=-1$;
the other one is
$\mathbf{a}^{0}=\infty_{-}( p_{0})$ and $\beta=1$.
This completes the proof.
\end{proof}

\begin{proof}[{Proof of Theorem \protect\ref{Hyper} }]
Obviously, Theorem 4.1-(i) follows from Propositions \ref{prop5-0}-\ref{prop5-3} easily.

(ii). The above argument shows that $\lim_{A\to\infty}i_{{\bf n}, p}(A,\pm W)=\infty_{\pm}(p)$ (by renaming $W,-W$ if necessary). As pointed out in (\ref{eqgamma}), $\overline{\Gamma_{\mathbf{n,}p}}\backslash
\Gamma_{\mathbf{n,}p}$ consists of two points $\infty_{\pm}:=\lim_{A\to\infty}(A,\pm W)$. Thus the embedding $i_{\mathbf{n,}p}$ has a natural extension to
$\bar{\imath}_{\mathbf{n,}p}:\overline{\Gamma_{\mathbf{n,}p}}$ $\rightarrow
\overline{Y_{\mathbf{n},p}}$ by defining $\bar{\imath
}_{\mathbf{n,}p}(\infty_{\pm}) : =\infty_{\pm}(p)$.
\end{proof}

We have the following corollary.

\begin{corollary}\label{coro4-4} The map
$\sigma_{\mathbf{n,}p}$ can be extended to be a finite morphism (still denoted by $\sigma_{\mathbf{n,}p}$) from the irreducible curve $\overline{%
\Gamma_{\mathbf{n,}p}}$ to $E_{\tau}$ such that%
\begin{equation*}
\underset{(A,\pm W) \to\infty_{\pm}}{\lim}\sigma _{\mathbf{%
n,}p}( A,\pm W) =\pm [p].
\end{equation*}
\end{corollary}

\section{The limiting of $Y_{\mathbf{n},p}(\tau) $ as $%
p\rightarrow \frac{\omega_{k}}{2}$}

\label{limiting}

\subsection{The counterpart of the above theory for $H({\bf n},B,\tau)$}
Before stating the main result of this section, we recall the following generalized Lam\'{e} equation (denoted by $H({\bf n},B,\tau)$)
\begin{equation}  \label{eq21-5}
y^{\prime \prime }(z)=I_{\mathbf{n}}(z;B,\tau)y(z),\quad z\in\mathbb{C}.
\end{equation}
where
\begin{equation}\label{eq21-6}
I_{\mathbf{n}}(z;B,\tau):=\sum_{k=0}^3n_k(n_k+1)\wp(z+\tfrac{\omega_k}{2}|\tau)+B,
\end{equation}
and $n_k\in\mathbb{Z}_{\geq 0}$ for all $k$.

There is a counterpart of the theory about GLE$({\bf n},p,A,\tau)$ established in previous sections for $H({\bf n},B,\tau)$, and the proof is simpler due to the absence of singularities $\pm p\notin E_{\tau}[2]$. Therefore, we only write down the conclusions without any details of the proofs. Remark that part of the statements listed below can be found in \cite{GW1,Takemura1,Takemura3} and references therein. In this section, we denote $\hat{N}:=\sum_{k}n_k$.

(i) Any solution of $H({\bf n},B,\tau)$ is meromorphic in $\mathbb{C}$. The corresponding second symmetric product equation (\ref{eq23}) has a unique even elliptic solution $\hat{\Phi}_{e}(z;B)$ expressed by
\begin{equation}
\hat{\Phi}_{e}(z;B)=\hat{C}_{0}(B)  +\sum_{k=0}^{3}\sum
_{j=0}^{n_{k}-1}\hat{b}_{j}^{(k)}(B)\wp(z+\tfrac{\omega_{k}}{2})^{n_{k}-j}
\label{3rd1}%
\end{equation}
where $\hat{C}_{0}(B), \hat{b}_{j}^{(k)}(B)$ are all polynomials in $B$ with $\deg \hat{C}_{0}>\max_{j,k} \deg\hat{b}_{j}^{(k)}$ and the leading coefficient of $\hat{C}_{0}(B)$
being $\frac{1}{2}$. Moreover, $\hat{\Phi}_{e}(z;B)=\hat{y}_1(z;B)$ $\hat{y}_1(-z;B)$, where $\hat{y}_1(z;B)$ is a common eigenfunction of the monodormy matrices of $H({\bf n},B,\tau)$ and up to a constant, can be written as
\begin{equation}\label{yby}\hat{y}_1(z;B)=y_{\bf a}(z):=\frac{e^{c({\bf a})z}\prod_{i=1}^{\hat{N}}\sigma(z-a_i)}{\prod_{k=0}^3\sigma(z-\frac{\omega_k}{2})^{n_k}}.\end{equation}

(ii) Let $\hat{W}$ be the Wroskian of $\hat{y}_1(z;B)$ and $\hat{y}_1(-z;B)$ (and denote $\hat{y}_1(z;B)$ by $\hat{y}_1(z;B,\hat{W})$), then $\hat{W}^2=Q_{\bf n}(B;\tau)$, where
$$Q_{\bf n}(B;\tau):=\hat{\Phi}_{e}'(z;B)^2-2\hat{\Phi}_{e}(z;B)\hat{\Phi}_{e}^{\prime \prime }(z;B)+4I_{\mathbf{n}%
}(z;B,\tau)\hat{\Phi}_{e}(z;B)^2$$
is a polynomial in $B$ with \emph{odd degree} and independent of $z$. Define the hyperelliptic curve $\Gamma_{\bf n}(\tau)$ by
\[\Gamma_{\bf n}(\tau):=\{(B, \hat{W})\,|\,\hat{W}^2=Q_{\bf n}(B;\tau)\}.\]
Then the map $i_{\bf n}: \Gamma_{\bf n}(\tau)\to \text{Sym}^{\hat{N}} E_{\tau}$ defined by
\[i_{\bf n}(B, \hat{W}):=\{[a_1],\cdots, [a_{\hat{N}}]\}\]
is an embedding, where $\{[a_1],\cdots, [a_{\hat{N}}]\}$ is uniquely determined by $\hat{y}_1(z;B,\hat{W})$ via (\ref{yby}). Let $Y_{{\bf n}}(\tau)$ be the image of $\Gamma_{{\bf n}}(\tau)$ in Sym$^{\hat{N}}E_\tau$
under $i_{{\bf n}}$, and $X_{{\bf n}}(\tau)$ be the image of $\{(B,\hat{W})\in\Gamma_{{\bf n}}|\hat{W}\neq
0\}$ under $i_{{\bf n}}$.

(iii) $\overline{X_{\mathbf{n}}(  \tau )  }=\overline
{Y_{\mathbf{n}}(  \tau )  }=Y_{\mathbf{n}}(\tau)
\cup\{\infty_0\}  $ where
\begin{equation}\label{infty0}
\infty_0:=\bigg(  \overset{n_{0}}{\overbrace{0,\cdot \cdot \cdot,0}}%
,\overset{n_{1}}{\overbrace{\frac{\omega_{1}}{2},\cdot \cdot \cdot,\frac
{\omega_{1}}{2}}},\overset{n_{2}}{\overbrace{\frac{\omega_{2}}{2},\cdot
\cdot \cdot,\frac{\omega_{2}}{2}}},\overset{n_{3}}{\overbrace{\frac{\omega_{3}%
}{2},\cdot \cdot \cdot,\frac{\omega_{3}}{2}}}\bigg)  .
\end{equation}
Furthermore, the map $i_{\mathbf{n}}:\Gamma_{\mathbf{n}}\rightarrow
Y_{\mathbf{n}}$ has a natural extension to $\bar{\imath}_{\mathbf{n}%
}:\overline{\Gamma_{\mathbf{n}}}\rightarrow \overline{Y_{\mathbf{n}}\left(
\tau \right)  }$ by $\bar{\imath}_{\mathbf{n}}(
\infty):=\infty_0$ where $\{\infty\}:=\overline{\Gamma_{\mathbf{n,}p}}%
\backslash \Gamma_{\mathbf{n,}p}$.

(iv)
The embedding $i_{\mathbf{n}}$ also induces the addition map $\sigma
_{\mathbf{n}}:\Gamma_{\mathbf{n}}\rightarrow E_{\tau}$ by
\[
\sigma_{\mathbf{n}}(  B,\hat{W}) : =\sum_{i=1}^{\hat{N}}[
a_{i}]  -\sum_{k=1}^{3}[\tfrac{n_{k}\omega_{k}}{2}].%
\]
Then $\sigma_{\mathbf{n}}$ can be extended as a finite morphism from the irreducible curve
$\overline{\Gamma_{\mathbf{n}}(\tau)}$ to $E_{\tau}$ such that%
\begin{equation}\label{bwlimit}
\underset{(  B,\hat{W})  \rightarrow \infty}{\lim}\sigma
_{\mathbf{n}}(  B,\hat{W})  =0.
\end{equation}

\subsection{The limit of $Y_{{\bf n}, p}$ as $p\to \frac{\omega_k}{2}$}
For any $\mathbf{n}=(n_0,n_1,n_2,n_3),$ we define $\mathbf{n}_{k}^{\pm}:=
( n_{k,0}^{\pm},n_{k,1}^{\pm},n_{k,2}^{\pm },n_{k,3}^{\pm}) $, where $%
n_{k,i}^{\pm}=n_{i}$ for $i\not =k$ and $n_{k,k}^{\pm}=n_{k}\pm1$. Recall $N=\sum_k n_k+1$. In the
following, we always identify a point $\mathbf{a}=\{[ a_{1}]
,\cdot \cdot \cdot,[ a_{N-2}] \}\in \overline{Y_{\mathbf{n}%
_{k}^{-}}( \tau ) }$ by $\mathbf{a}=\{[ a_{1}]
,\cdot \cdot \cdot,[ a_{N-2}] ,[\frac{\omega_{k}}{2}],[\frac{%
\omega_{k}}{2}]\}$.

The main result of this section is the following theorem:

\begin{theorem}
\label{cor7}Let $\mathbf{n}=(n_{0},n_{1},n_{2},n_{3})$ with $n_k\geq 0$ for all $k$. Then%
\begin{equation*}
\lim_{p\to \frac{\omega_{k}}{2}}\overline{Y_{\mathbf{n},p}(
\tau) }=\overline{Y_{\mathbf{n}_{k}^{+}}( \tau ) }\cup
\overline{Y_{\mathbf{n}_{k}^{-}}( \tau ) }.
\end{equation*}
\end{theorem}

In the following, we will give the complete proof for the case $p\rightarrow
0$. For other cases the proof is similar. Let $p$ be a sequence of points
in $E_{\tau}$ and a sequence $( A( p) ,W( p)
)\in\Gamma_{\mathbf{n},p}$ with the corresponding $%
\{[a_{1}( p) ],\cdot \cdot \cdot,[a_{N}( p) ]\}\in Y_{\mathbf{n},p}$ satisfying $%
([a_{1}( p) ],\cdot \cdot \cdot,[a_{N}( p) ])
\to ([a_{1}^{0}],\cdot \cdot \cdot,[a_{N}^{0}])$. We assume that up to a subsequence if
necessary, $A( p) \rightarrow A_{0}$ $\in \mathbb{C\cup}\{
\infty \} $ as $p\to 0$. Then the same proof as Proposition \ref{prop5-0} implies
\begin{equation}
\begin{array}{l}
I_{\mathbf{n}}( z;p,A(p),\tau) \text{ converges } \text{if and only if
}c( \mathbf{a}( p) ) \text{ converges} \\
\text{if and only if }B( p) \text{
converges as }p\rightarrow 0.%
\end{array}
\label{6-1}
\end{equation}
Indeed, we have a more precise statement. Recall $e_k:=\wp(\frac{\omega_k}{2})$ for $k=1,2,3$.

\begin{proposition}
\label{prop11}$B(p) $ is convergent as $p\rightarrow0$ if and only if
\begin{equation}\label{app}A(p)=A^\pm(p):=\alpha_{0}^{\pm}p^{-1}+\alpha_1^\pm p+o(p)\end{equation} with
\begin{equation}\label{alph}\alpha_0^{+}:=-(\tfrac{1}{4}+n_0),\quad \alpha_0^{-}:=\tfrac{3}{4}+n_0\end{equation} and some $\alpha_1^{\pm}\in \mathbb{C}$. Furthermore, if (\ref{app}) holds, then $%
I_{\bf n}(z;p,A^\pm(p),\tau)$ converges to $I_{{\bf n}_{0}^\pm}(z;B^\pm,\tau)$, where $B^\pm$ and $\alpha_1^\pm$
are related by
\begin{equation}  \label{qe}
B^\pm=\mp(2n_0+1)\alpha^{\pm}_1-\sum\limits_{k=1}^3n_k(n_k+1)e_k.
\end{equation}
\end{proposition}

\begin{proof}
Recall (\ref{101}) and it yields
{\allowdisplaybreaks
\begin{align}
B(p) & =A(p)^{2}-\zeta( 2p) A(p)-\frac{3}{4}\wp(
2p) -\sum_{k=0}^{3}n_{k}( n_{k}+1) \wp( p+\tfrac{%
\omega_{k}}{2})  \label{74} \\
& =\frac{1}{p^{2}}\left( ( A(p)p) ^{2}-\frac{A(p)p}{2}-\frac{3}{16}%
-n_{0}( n_{0}+1) \right)  \notag \\
&\quad -\sum_{k=1}^{3}n_{k}( n_{k}+1) e_{k}+o( 1)  \notag\\
&=\frac{(A(p)p-\alpha_0^+)(A(p)p-\alpha_0^-)}{p^{2}} -\sum_{k=1}^{3}n_{k}( n_{k}+1) e_{k}+o( 1),\nonumber
\end{align}
}%
where $\alpha_{0}^{\pm}$ is defined in (\ref{alph}).
From here, we easily see that $B(p) $ converges as $p\to 0$ if and only if $A(p)$ satisfies (\ref{app}).

Now suppose (\ref{app}) holds, then (\ref{74}) implies $B(p)\to B^{\pm}$ as $p\to 0$. Furthermore,
\begin{align*}
&n_0(n_0+1)\wp(z) +\tfrac{3}{4}%
(\wp(z+p)+\wp(z-p))+A(p)(\zeta(z+p)-\zeta(z-p))\\
&=n_0(n_0+1)\wp(z) +\tfrac{3}{2}\wp(z)-2A(p)p\wp(z)+o(1)\\
&\to [n_0(n_0+1)+\tfrac{3}{2}-2\alpha_0^{\pm}]\wp(z)=n_{0,0}^{\pm}(n_{0,0}^{\pm}+1)\wp(z),
\end{align*}
where $n_{0,0}^{\pm}=n_{0}\pm 1$. Thus $%
I_{\bf n}(z;p,A^\pm(p),\tau)$ converges to $I_{{\bf n}_{0}^\pm}(z;B^\pm,\tau)$.\end{proof}

Now, we discuss the limit $\mathbf{a}^{0}$ of $\mathbf{a}(p) $
as $p\rightarrow0$. We divide two cases to discuss: (i) $B(p) \rightarrow
B_{0}\in \mathbb{C}$, and (ii) $B(p) \rightarrow
\infty$.

\begin{proposition}
\label{prop3}Suppose $B(p) \rightarrow
B_{0}\in \mathbb{C}$ as $p\rightarrow0$. Then either

\begin{itemize}
\item[(i)] $\mathbf{a}^{0}\in Y_{\mathbf{n}^{+}}( \tau ) $, where
$\mathbf{n}^{+}=\mathbf{n}_0^{+}=( n_{0}+1,n_{1},n_{2},n_{3}) $; or

\item[(ii)] $\mathbf{a}^{0}=\{[ a_{1}^{0}] ,\cdot \cdot \cdot,%
[ a_{N-2}^{0}] ,0,0\}$ and $\{ [ a_{1}^{0}] ,\cdot
\cdot \cdot,[ a_{N-2}^{0}] \} \in Y_{\mathbf{n}^{-}}(
\tau) $, where $\mathbf{n}^{-}=\mathbf{n}_0^{-}=(
n_{0}-1,n_{1},n_{2},n_{3}) $.
\end{itemize}
\end{proposition}

\begin{proof}
By Proposition \ref{prop11}, we see that $I_{\bf n}(z;p,A(p),\tau)$ converges to either (i) $I_{{\bf n}^+}(z;B^+,\tau)$
or (ii) $I_{{\bf n}^-}(z;B^-,\tau)$.

For the case (i), because $c( \mathbf{a}( p))
\rightarrow c_0$ for some $c_0\in\mathbb{C}$ by (\ref{6-1}), we
see that
\begin{equation}\label{yfff}
y_{\mathbf{a}( p),c( \mathbf{a}( p)) }( z;p)
\rightarrow y_{\mathbf{a}^{0}}(z) :=\frac{e^{c_0z}\prod_{j=1}^{N}\sigma( z-a_{j}^{0}) }{%
\sigma( z)^{n_{0}+1} \prod_{k=1}^{3}\sigma
( z-\frac{\omega_{k}}{2}) ^{n_{k}}},
\end{equation}
and $y_{\mathbf{a}^{0}}(z) $ is a solution to $H(\mathbf{n}%
^{+},B^{+},\tau)$. This implies $\mathbf{a}^{0}\in$ $Y_{\mathbf{n}%
^{+}}( \tau ) $.

For the case (ii), we also have (\ref{yfff}) but
$y_{\mathbf{a}^{0}}(z) $ is a solution to $H(\mathbf{n}%
^{-},B^{-},\tau)$. It follows from (\ref{yby}) that at least two of $\{[ a_{1}^{0}] ,\cdot
\cdot \cdot,[ a_{N}^{0}]\} $ must be $0$. After a
rearrangement of the index, we have might assume
\begin{equation*}
\mathbf{a}^{0}=\{[ a_{1}^{0}] ,\cdot \cdot \cdot,[
a_{N-2}^{0}] ,0,0\}.
\end{equation*}
Then we have $\{[ a_{1}^{0}] ,\cdot \cdot \cdot,[
a_{N-2}^{0}]\}\in Y_{\mathbf{n}^{-}}( \tau) $.
\end{proof}

\begin{remark}
In the following, we always identify a point $\{[ a_{1}^{0}] ,\cdot \cdot \cdot,[
a_{N-2}^{0}]\}\in Y_{\mathbf{n}^{-}}( \tau) $ with the point $\{[ a_{1}^{0}] ,\cdot \cdot \cdot,[
a_{N-2}^{0}] ,0,0\}$ in Sym$^NE_{\tau}$.
\end{remark}

Next, we consider the case $I_{\mathbf{n}}( z;p,A( p),\tau) $
diverges, i.e. $B( p) \rightarrow \infty$ as $p\rightarrow0$.

\begin{proposition}
\label{prop4}Suppose $B( p) \rightarrow \infty
$ as $p\rightarrow0$. Then
\begin{equation*}
\mathbf{a}(p) \rightarrow \mathbf{a}^{0}=\infty(0)
\end{equation*}
where
\begin{equation}\label{innn}
\infty(0) :=\bigg( \overset{n_{0}}{\overbrace{0,\cdot \cdot
\cdot,0}},\overset{n_{1}}{\overbrace{\frac{\omega_{1}}{2},\cdot \cdot \cdot,%
\frac{\omega_{1}}{2}}},\overset{n_{2}}{\overbrace{\frac{\omega_{2}}{2},\cdot
\cdot \cdot,\frac{\omega_{2}}{2}}},\overset{n_{3}}{\overbrace {\frac{%
\omega_{3}}{2},\cdot \cdot \cdot,\frac{\omega_{3}}{2}}},0\bigg) .
\end{equation}
\end{proposition}

\begin{proof}
Since $B(p) \rightarrow \infty$, we see from (\ref{6-1}) that $%
I_{\bf n}(z;p,A(p),\tau)$ diverges and $c( \mathbf{a}(p)) \rightarrow
\infty$ as $p\rightarrow0$. By differentiating both expressions \eqref{5051} and \eqref{72} of $I_{\bf n}(z; p, A(p), \tau)$ and considering $p\to 0$, we obtain
\begin{align}
& 2c( \mathbf{a}(p)) E( z;\mathbf{a}(p), p) +2D( z;\mathbf{a}(p), p)E( z;\mathbf{a}(p), p)+E'( z;\mathbf{a}(p), p)  \label{D3} \\
& =\sum_{k=0}^{3}n_{k}( n_{k}+1) \wp^{\prime}( z+\tfrac{\omega_{k}}{2}) +\frac{3}{4}( \wp^{\prime}( z+p)
+\wp^{\prime}( z-p))  \notag \\
&\quad +A(p) [ -2\wp^{\prime}( z) p+O(
p^{3})] .  \notag
\end{align}
Then we divide our discussion into two cases.

{\bf Case 1}. Suppose $A(p) p\rightarrow \beta \in \mathbb{C}$ up to a subsequence. Then the
RHS of (\ref{D3}) is uniformly convergent outside the singularties $E_{\tau }%
[2] $. Since $c( \mathbf{a}(p)) \rightarrow \infty$, we
have
\begin{equation*}
E( z;\mathbf{a}(p), p) \rightarrow0\text{ uniformly outside }E_{\tau }%
[2].
\end{equation*}
So we see from (\ref{79}) that
\begin{equation}\label{eqqqq}
\sum_{k=0}^{3}n_{k}\wp( z-\tfrac{\omega_{k}}{2})
-\sum_{j=1}^{N}\wp( z-a_{j}^{0}) +\wp ( z) =0.
\end{equation}
This implies that there are $( n_{0}+1) $ of $a_{j}^{0}$ equal to
$0$ and $n_{k}$ of $a_{j}^{0}$ equal to $\frac{\omega_{k}}{2}$ for $k\in \{1,2,3\}$, namely ${\bf a}^0=\infty(0)$.

(ii) Suppose $A(p)p\to \infty$. By dividing $A(
p) p$ on both sides of (\ref{D3}), we see that%
\begin{align} \label{D4}
\frac{2c( \mathbf{a}(p)) }{A( p) p} & E( z;\mathbf{a}(p), p) =\frac{-2D( z;\mathbf{a}(p), p)E( z;\mathbf{a}(p), p)-E'( z;\mathbf{a}(p), p)}{A(
p) p} \notag\\
& +\frac{1}{A( p) p}\left[
\begin{array}{c}
\sum_{k=0}^{3}n_{k}( n_{k}+1) \wp^{\prime}( z+\tfrac {%
\omega_{k}}{2}) \\
+\frac{3}{4}( \wp^{\prime}( z+p) +\wp^{\prime}(
z-p) )%
\end{array}
\right]   -2\wp^{\prime}( z) +O( p^{2}) .
\end{align}
Since the RHS of (\ref{D4}) converges uniformly to $-2\wp^{\prime}(
z) $ outside $E_{\tau}[2] $, we have either $\frac{c(
\mathbf{a}(p)) }{A( p) p}\rightarrow \beta \in \mathbb{C}$ or $%
\frac{c( \mathbf{a}(p)) }{A( p) p}$ $\rightarrow \infty$.
Suppose $\frac{c(
\mathbf{a}(p)) }{A( p) p}\rightarrow
\beta \in \mathbb{C}$. Letting $p\to 0$, it follows from (\ref{D4}) and (\ref{79}) that
\begin{equation}
\beta \bigg[\sum_{k=0}^{3}n_{k}\wp( z-\tfrac{\omega_{k}}{2})
-\sum_{j=1}^{N}\wp( z-a_{j}^{0}) +\wp ( z)\bigg]
=-\wp^{\prime}(z) .   \label{D5}
\end{equation}
But the RHS of (\ref{D5}) has a pole of order $3$ at $z=0$ while the LHS does not, a contradiction. Thus, $%
\frac{c( \mathbf{a}(p)) }{A(p) p}$ $\rightarrow \infty$
and then we obtain (\ref{eqqqq}) again, which gives $\mathbf{a}^{0}=\infty(0)$.
This completes the proof.
\end{proof}

\begin{corollary}
\label{u0}Let $( A( p) ,W( p) ) \in \Gamma_{%
\mathbf{n},p}$. Assume that $\sigma_{\mathbf{n},p}( A( p)
,W( p) )\rightarrow [u_{0}]\not =0$ as $
p\rightarrow0$. Then $I_{\mathbf{n}}( z;p,A( p),\tau) $
converges as $p\rightarrow0$.
\end{corollary}

\begin{proof}
Let $\mathbf{a}( p)=i_{{\bf n},p}(A( p) ,W( p)) $ be the corresponding point in $Y_{\mathbf{n}%
,p}( \tau) $ and we may assume $\mathbf{a}( p)\rightarrow
\mathbf{a}^{0}$. Suppose $I_{\mathbf{n}}( z;p,A( p),\tau) $ diverges as $p\rightarrow0$. Then Proposition \ref{prop4} says
that $\mathbf{a}^{0}=\infty( 0)$, which implies $[u_{0}]=0$,
a contradiction.
\end{proof}

Now we are
ready to prove Theorem \ref{cor7}.

\begin{proof}[Proof of Theorem \protect\ref{cor7}]
We prove the case for $p\rightarrow0$ and the other three cases are similar.
Let $p\rightarrow0$. By Propositions \ref{prop3} and \ref{prop4}, for any $%
\mathbf{a}( p) $ $\in \overline{Y_{\mathbf{n},p}( \tau
) }$, $\mathbf{a}( p) \rightarrow \mathbf{a}^{0}$ $\in
\overline{Y_{\mathbf{n}^{+}}( \tau ) }\cup \overline{Y_{\mathbf{n}%
^{-}}( \tau ) }$ as $p\rightarrow0$. Conversely, we want to show
that any point $\mathbf{a}\in \overline{Y_{\mathbf{n}^{+}}( \tau
) }\cup \overline {Y_{\mathbf{n}^{-}}( \tau ) }$ could be a
limit point of some $\mathbf{a}( p)\in \overline{Y_{\mathbf{n}%
,p}( \tau ) }$ as $p\rightarrow0$.

First we note that $\infty(0)$ is the limit of $\infty_\pm(p)$ as $%
p\rightarrow0$.

Given any $\mathbf{a}=\{[a_1],\cdots,[a_N]\}\in Y_{{\bf n}^+}(\tau)$ and let $i_{{\bf n}^+}(B,W)=%
\mathbf{a}$ for some $B$. For $p$ close to $0$, we set $A(p):=-(\frac{1}{4}%
+n_0)p^{-1}+\alpha_1^{+} p$, where $\alpha_1^{+}$ is given by (\ref{qe}): \[
\alpha_1^{+}:=-\frac{B+\sum_{k=1}^3n_k(n_k+1)e_k}{2n_0+1},\] and $B(p)$ by
(\ref{101}). Let $(A(p),\pm W(p))\in\Gamma_{{\bf n},p}$ and $\pm\mathbf{a}%
(p)=i_{{\bf n},p}(A(p),\pm W(p)).$ By Propositions \ref{prop11}-\ref{prop3}, $I_{\bf n}(z;p,A(p),\tau)$ converges
to $I_{{\bf n}^+}(z;B,\tau)$ and so $\{\pm\mathbf{a}(p)\}$ converges to $\{\pm\mathbf{a}\}$. By a
similar argument, we could prove that any point of $Y_{{\bf n}^-}(\tau)$ can be
approximated by $Y_{{\bf n},p}(\tau)$ as $p\to 0$. This completes the proof.
\end{proof}

\section{The degree of the addition map}

\label{degreeof}

In previous sections, we have defined the addition map $\sigma_{\mathbf{n,}p}(
\cdot|\tau ) $ and $\sigma_{\mathbf{n}}( \cdot|\tau ) $
from $\overline{\Gamma_{\mathbf{n},p}( \tau )} $ and $\overline {%
\Gamma_{\mathbf{n}}( \tau )} $ onto $E_{\tau}$ respectively.
Since $\overline{\Gamma_{\mathbf{n},p}( \tau )} $ and $\overline{%
\Gamma_{\mathbf{n}}( \tau )} $ are irreducible algebraic curves,
it is an elementary fact that the degrees of $\sigma _{\mathbf{n,}p}(
\cdot|\tau ) $ and $\sigma_{\mathbf{n}}( \cdot|\tau) $ are
well-defined. The purpose of this section is to prove the following result.

\begin{theorem}
\label{degreeformula}Let $\mathbf{n}=(n_{0},n_{1},n_{2},n_{3})$ with $%
n_{k}\in \mathbb{Z}_{\geq0}$ and $p\in E_{\tau}\backslash E_{\tau}[2]$. Then
\begin{equation}
\deg \sigma_{\mathbf{n}}=\sum_{k=0}^{3}\frac{n_{k}(n_{k}+1)}{2},
\label{degh}
\end{equation}
and%
\begin{equation}
\deg \sigma_{\mathbf{n,}p}=\sum_{k=0}^{3}n_{k}(n_{k}+1)+1.   \label{degp}
\end{equation}
\end{theorem}

To compute the degree of the map $\sigma_{\mathbf{n},p}$, we
consider $\wp(\sigma_{\mathbf{n,}p}(A,W)|\tau)$. Since $\sigma_{\mathbf{n,}%
p}(A,-W)=-\sigma_{\mathbf{n,}p}(A,W)$, $%
\wp(\sigma_{\mathbf{n,}p}(A,W)|\tau)$ depends on $A$ only and is a
meromorphic function of $A\in\mathbb{C}$. Together with the facts that (i) $%
\wp(\sigma_{\mathbf{n,}p}(\cdot,W)|\tau)$ has finitely many poles $A$'s and
(ii) $\sigma_{\mathbf{n,}p}(A,W)\rightarrow \pm [p]$ by Corollary \ref{coro4-4} and so
\begin{equation}
\wp(\sigma_{\mathbf{n,}p}(A,W)|\tau)\rightarrow \wp(p|\tau)\text{\ as }%
A\rightarrow \infty \text{ for any }\tau \text{ and }p\not \in E_{\tau}[2],
\label{p-p}
\end{equation}
we conclude that there are \emph{coprime polynomials} $P_{j}(A,p;\tau )\in
\mathbb{C}[A]$, $j=1,2$, such that%
\begin{equation}
\wp(\sigma_{\mathbf{n,}p}(A,W)|\tau)=\frac{P_{1}(A,p;\tau)}{P_{2}(A,p;\tau)}%
.   \label{pu}
\end{equation}
Indeed, a more delicate argument shows $P_{j}(A,p;\tau)\in \mathbb{Q[}%
e_{k}(\tau),\wp(p|\tau),\wp^{\prime}(p|\tau)]$ $[A]$; see e.g. \cite[%
Proposition 3.1]{Takemura}. Here $e_k(\tau):=\wp(\frac{\omega_k}{2}|\tau)$, $k=1,2,3$. Write%
\begin{equation*}
P_{1}(A,p;\tau)=\sum_{k=0}^{m_{1}}a_{k}(p;\tau)A^{k}\text{ \ and }%
P_{2}(A,p;\tau)=\sum_{k=0}^{m_{2}}b_{k}(p;\tau)A^{k},
\end{equation*}
where $a_{k}(p;\tau),b_{k}(p;\tau)\in \mathbb{Q[}e_{k}(\tau),\wp(p|\tau
),\wp^{\prime}(p|\tau)]$ with $a_{m_{1}}(p;\tau)\not \equiv 0$ and $%
b_{m_{2}}(p;\tau)\not \equiv 0$.

\begin{lemma}
\label{lem56} Let $\tau \in \mathbb{H}$ and $p\not \in E_{\tau}[2]$. Then
under the above notations, the followings hold.

\begin{itemize}
\item[(1)] $m_{1}=m_{2}=:m$ and%
\begin{equation}
\frac{a_{m}(p;\tau)}{b_{m}(p;\tau)}=\wp(p|\tau).   \label{a-b-k}
\end{equation}

\item[(2)] If $b_{m}(p;\tau)\not =0$, then%
\begin{equation}
\deg_{A}P_{1}(A,p;\tau)\leq \deg_{A}P_{2}(A,p;\tau)=\deg \sigma_{\mathbf{n,}%
p}=m.   \label{p1p2}
\end{equation}
Furthermore,%
\begin{equation}
\deg_{A}P_{1}(A,p;\tau)=\deg_{A}P_{2}(A,p;\tau)\text{ }\Leftrightarrow \text{
}\wp(p|\tau)\not =0.   \label{p1p2-1}
\end{equation}
\end{itemize}
\end{lemma}

\begin{proof}
(1) Take any $\tau$ and $p\not \in E_{\tau}[2]$ such that $\wp(p|\tau)\not =0
$, $a_{m_{1}}(p;\tau)\not =0$ and $b_{m_{2}}(p;\tau)\not =0$. Then it
follows from (\ref{p-p}) and (\ref{pu}) that $m_{1}=m_{2}$ (denote it by $m$%
) and (\ref{a-b-k}) holds. Since both sides of (\ref{a-b-k}) are meromorphic
in $(\tau,p)$, we conclude that (\ref{a-b-k}) holds for all $\tau$ and $%
p\not \in E_{\tau}[2]$.

(2) Fix any $\tau$ and $p\not \in E_{\tau}[2]$ such that $b_{m}(p;\tau )\not
=0$. Then $\deg_{A}P_{2}=m\geq \deg_{A}P_{1}$ and (\ref{p1p2-1}) follows
from (\ref{a-b-k}). It suffices to prove $\deg \sigma_{\mathbf{n,}p}=m$.
Denote $\tilde{m}:=\deg \sigma_{\mathbf{n,}p}$. Take any $\sigma_{0}\not \in
E_{\tau }[2]\cup \{ \pm \lbrack p]\}$ such that
\begin{equation}
\sigma_{\mathbf{n,}p}^{-1}\left( \sigma_{0}\right) =\{ \left(
A_{1},W_{1}\right) ,\cdot \cdot \cdot,\left( A_{\tilde{m}},W_{\tilde{m}%
}\right) \}   \label{po1p1}
\end{equation}
consists of $\tilde{m}$ distinct elements in $\Gamma_{\mathbf{n},p}\left(
\tau \right) $ and the polynomial%
\begin{equation}
P_{1}(\cdot,p;\tau)-\wp(\sigma_{0}|\tau)P_{2}(\cdot,p;\tau)\text{ has only
simple zeros.}   \label{po1}
\end{equation}

If $W_{i}=0$ for some $i$, then it follows from Theorem \ref{thm3} that $%
\sigma_{0}=\sigma_{\mathbf{n,}p}(A_{i},W_{i})\in E_{\tau}[2]$, a
contradiction. Hence $W_{i}\not =0$ for all $i$. If $A_{i}=A_{j}$ for some $%
i\not =j$, then $W_{i}$ $=$ $-W_{j}$. By using%
\begin{equation*}
\sigma_{0}=\sigma_{\mathbf{n,}p}(A_{j},W_{j})=\sigma_{\mathbf{n,}%
p}(A_{i},-W_{i})=-\sigma_{\mathbf{n,}p}( A_{i},W_{i})
=-\sigma_{0}\text{ in}\; E_{\tau},
\end{equation*}
we have $\sigma_{0}\in E_{\tau}[2]$, a contradiction again. Thus these $A_{i}
$'s are $\tilde{m}$ distinct roots of (\ref{po1}), namely
\begin{equation}
\deg_{A}(P_{1}(A,p;\tau)-\wp(\sigma_{0}|\tau)P_{2}(A,p;\tau))\geq \tilde{m}.
\label{qqq}
\end{equation}

Assume by contradiction that $\deg_{A}(P_{1}-\wp(\sigma_{0}|\tau)P_{2})>%
\tilde{m}$. By (\ref{po1}), the polynomial $P_{1}-\wp(\sigma_{0}|\tau)P_{2}$
has another zero $A_{\tilde{m}+1}\not \in \{A_{1},\cdot \cdot \cdot,A_{%
\tilde {m}}\}$. Then the points $\left( A_{\tilde{m}+1},\pm W_{\tilde{m}%
+1}\right) \in \Gamma_{\mathbf{n},p}$ satisfy%
\begin{equation*}
\wp(\sigma_{\mathbf{n,}p}(A_{\tilde{m}+1},\pm W_{\tilde{m}+1})|\tau )=\frac{%
P_{1}(A_{\tilde{m}+1},p;\tau)}{P_{2}(A_{\tilde{m}+1},p;\tau)}%
=\wp(\sigma_{0}|\tau),
\end{equation*}
i.e. either $\sigma_{\mathbf{n,}p}( A_{\tilde{m}+1},W_{\tilde{m}%
+1}) =\sigma_{0}$ or $\sigma_{\mathbf{n,}p}( A_{\tilde{m}+1},-W_{%
\tilde{m}+1}) =\sigma_{0}$, which is a contradiction with (\ref{po1p1}%
) and $A_{\tilde{m}+1}\not \in \{A_{1},\cdot \cdot \cdot ,A_{\tilde{m}}\}$.
This proves%
\begin{equation*}
\deg_{A}(P_{1}(A,p;\tau)-\wp(\sigma_{0}|\tau)P_{2}(A,p;\tau))=\tilde{m}
\end{equation*}
hold for almost $\sigma_{0}\not \in E_{\tau}[ 2] \cup \{ \pm
\lbrack p]\}$, which implies%
\begin{equation*}
m=\max \{ \deg_{A}P_{1},\deg_{A}P_{2}\}=\tilde{m}=\deg \sigma_{\mathbf{n,}%
p}.
\end{equation*}
The proof is complete.
\end{proof}

\begin{lemma}
\label{lem58} Let $\tau_{0}\in \mathbb{H}$ and $p_{0}\not \in E_{\tau_{0}}[2]
$. If at least one of $\{a_{m}(p_{0};\tau_{0}),$ $\cdot \cdot \cdot,$ $%
a_{0}(p_{0};\tau_{0}),$ $b_{m}(p_{0};\tau_{0}),$ $\cdot \cdot \cdot,$ $%
b_{0}(p_{0};\tau_{0})\}$ is not zero, then $b_{m}(p_{0};\tau_{0})\not =0$.
\end{lemma}

\begin{proof}
Assume by contradiction that $b_{m}(p_{0};\tau_{0})=0$. By our assumption
and Lemma \ref{lem56}-(1), we can take $\sigma_{0}\not \in E_{\tau_{0}}[
2] \cup \{ \pm \lbrack p_{0}]\}$ such that
\begin{equation*}
P(A,p_{0};\tau_{0}):=P_{1}(A,p_{0};\tau_{0})-\wp(\sigma_{0}|%
\tau_{0})P_{2}(A,p_{0};\tau_{0})
\end{equation*}
is a nonzero polynomial with $\deg P(A,p_{0};\tau_{0})\leq m-1$. Let $%
\tau_{\ell},p_{\ell}\not \in E_{\tau_{\ell}}[2]$ such that $( p_{\ell
},\tau_{\ell}) $ $\rightarrow(p_{0},\tau_{0})$ as $\ell\to\infty$, $b_{m}(p_{\ell};\tau_{%
\ell})$ $\not =0$ and $\sigma_{0}\not \in E_{\tau_{\ell}}[ 2]
\cup \{ \pm \lbrack p_{\ell}]\}$ for all $\ell$. Then Lemma \ref{lem56}-(2)
gives $\deg \sigma_{\mathbf{n,}p_{\ell}}(\cdot|\tau_{\ell})=m$, so%
\begin{equation*}
\sigma_{\mathbf{n},p_{\ell}}^{-1}\left( \sigma_{0}|\tau_{\ell}\right) =\left
\{ \left( A_{\ell,1},W_{\ell,1}\right) ,\cdot \cdot \cdot,\left(
A_{\ell,m},W_{\ell,m}\right) \right \}
\end{equation*}
and $A_{\ell,i}$, $i=1,\cdot \cdot \cdot,m,$ are all the zeros of
\begin{equation*}
P(A,p_{\ell};\tau_{\ell}):=P_{1}(A,p_{\ell};\tau_{\ell})-\wp(\sigma_{0}|%
\tau_{\ell})P_{2}(A,p_{\ell};\tau_{\ell}).
\end{equation*}
Since $\deg P(A,p_{0};\tau_{0})\leq m-1$, there is some $i$ such that $%
A_{\ell,i}\rightarrow \infty$ as $\ell \rightarrow \infty$. Then Proposition \ref{prop5-3} implies that
the corresponding $(a_{1}(A_{\ell,i},p_{\ell},\tau_{\ell}),\cdot \cdot \cdot
,a_{N}(A_{\ell,i},p_{\ell},\tau_{\ell}))$ of $\left( A_{\ell,i},W_{\ell
,i}\right) $ converges to $\infty_{\pm}( p_{0}) $, which yields
\begin{equation*}
\sigma_{0}=\lim_{\ell \rightarrow \infty}\sum_{j=1}^{N}a_{j}( A_{\ell
,i},p_{\ell},\tau_{\ell}) -\sum_{k=1}^{3}\frac{n_{k}\omega_{k}}{2}=\pm
p_{0},
\end{equation*}
a contradiction to the choice of $\sigma_{0}$. This proves $%
b_{m}(p_{0};\tau_{0})$ $\not =0$.
\end{proof}

\begin{lemma}
\label{main thm4}$\deg \sigma_{\mathbf{n,}p}\left( \cdot|\tau \right) =m$ is
independent of $\tau \in \mathbb{H}$ and $p\not \in E_{\tau}[2]$.
\end{lemma}

\begin{proof}
By Lemma \ref{lem56}-(2), it suffices to prove that $\deg \sigma_{\mathbf{n,}%
p_{0}}(\cdot|\tau_{0})=m$ for any $\tau_{0}$, $p_{0}\not \in E_{\tau_{0}}[2]$
satisfying $b_{m}(p_{0};\tau_{0})=0$.

\textbf{Case 1.} At least one of the coefficients of $P_{i}(A,p;\tau_{0})$
is not identical $0$ in $p$.

Then we may divide $P_{i}(A,p;\tau_{0})$ by a common factor $(p-p_{0})^{\ell}
$ such that all the coefficients of the new $\tilde{P}_{i}\left( A,p;\tau
_{0}\right) :=P_{i}(A,p;\tau_{0})/(p-p_{0})^{\ell}$ are holomorphic at $p_{0}
$\ and at least one of them do not vanish at $p_{0}$. Then Lemma \ref{lem58}
implies $\tilde{b}_{m}(p_{0};\tau_{0})\not =0$, where $\tilde {b}%
_{m}(p;\tau):=b_{m}(p;\tau)/(p-p_{0})^{\ell}$, and so the same proof as Lemma \ref%
{lem56}-(2) shows that $\deg \sigma_{\mathbf{n,}p_{0}}(\cdot|\tau_{0})=m$.

\textbf{Case 2}. All the coefficients of $P_{i}(A,p;\tau_{0})$ are identical
$0$ in $p$.

Then we may divide $P_{i}(A,p;\tau)$ by a factor $(\tau-\tau_{0})^{\ell}$
such that Case 1 holds, which again implies that $\deg \sigma_{\mathbf{n,}%
p_{0}}(\cdot|\tau_{0})=m$. The proof is complete.
\end{proof}

For the addition map $\sigma_{\mathbf{n}}(\cdot|\tau)$ from $H(\mathbf{n}%
,B,\tau)$, we have the similar result: There are coprime polynomials $\hat {P%
}_{j}(B;\tau)\in \mathbb{Q[}e_{k}(\tau)][B]$ such that%
\begin{equation}
\wp(\sigma_{\mathbf{n}}(B,\hat{W})|\tau)=\frac{\hat{P}_{1}(B;\tau)}{\hat {P}%
_{2}(B;\tau)}.   \label{puu}
\end{equation}
Since (\ref{bwlimit}) says $\sigma_{\mathbf{n}}(B,\hat{W})$ $\rightarrow0$ as $B\rightarrow \infty
$, we have $\deg_{B}\hat{P}_{1}( B;\tau ) >\deg_{B}\hat{P}%
_{2}( B;\tau ) $. Takemura \cite[Proposition 3.2]{Takemura4}
proved that%
\begin{equation}
\deg_{B}\hat{P}_{1}( B;\tau ) =\deg_{B}\hat{P}_{2}( B;\tau) +1\text{ for any }\tau \in \mathbb{H},   \label{113}
\end{equation}%
\begin{equation}
\wp(\sigma_{\mathbf{n}}(B,\hat{W})|\tau)=\frac{4B}{%
[\sum_{k=0}^{3}n_{k}(n_{k}+1)]^{2}}+O(1)\text{ as }B\rightarrow \infty \text{
for fixed }\tau.   \label{114}
\end{equation}
Write%
\begin{equation*}
\hat{P}_{1}(B;\tau)=\sum_{k=0}^{\hat{m}_{1}}\hat{a}_{k}(\tau)B^{k}\text{ \ and }%
\hat{P}_{2}(B;\tau)=\sum_{k=0}^{\hat{m}_{2}}\hat{b}_{k}(\tau)B^{k},
\end{equation*}
where $\hat{a}_{k}(\tau),\hat{b}_{k}(\tau)\in \mathbb{Q[}e_{k}(\tau)]$ with $%
\hat{a}_{\hat{m}_{1}}(\tau)\not \equiv 0$ and $\hat{b}_{\hat{m}%
_{2}}(\tau)\not \equiv 0$. Then we have

\begin{lemma}
\label{cor ind}Under the above notations, the followings hold.

\begin{itemize}
\item[(1)] $\hat{m}_{2}+1=\hat{m}_{1}=:\hat{m}$ and%
\begin{equation}
\frac{\hat{a}_{\hat{m}}(\tau)}{\hat{b}_{\hat{m}-1}(\tau)}=\frac{4}{[\sum
_{k=0}^{3}n_{k}(n_{k}+1)]^{2}}=:C(\mathbf{n}).   \label{ab}
\end{equation}

\item[(2)] $\deg \sigma_{\mathbf{n}}( \cdot|\tau ) =\hat{m}$ is
independent of $\tau \in \mathbb{H}$. Furthermore,%
\begin{equation*}
\deg_{B}\hat{P}_{1}( B;\tau ) =\hat{m}=\deg \sigma_{\mathbf{n}%
}( \cdot|\tau) \text{ if }\hat{a}_{\hat{m}}(\tau)\not =0.
\end{equation*}
\end{itemize}
\end{lemma}

\begin{proof}
The proof is similar to those of Lemmas \ref{lem56}-\ref{main thm4} with
minor modifications; we omit the details here.
\end{proof}

Recall Theorem \ref{cor7} that%
\begin{equation*}
\lim_{p\rightarrow \frac{\omega_{k}}{2}}\overline{Y_{\mathbf{n},p}(
\tau ) }=\overline{Y_{\mathbf{n}_{k}^{+}}( \tau) }\cup
\overline{Y_{\mathbf{n}_{k}^{-}}( \tau ) }.
\end{equation*}
Thus it is reasonably expected that the degree of $\sigma_{\mathbf{n,}p}$
should be the sum of the degree of $\sigma_{\mathbf{n}_{k}^{+}}$ and $%
\sigma_{\mathbf{n}_{k}^{-}}$. The following result confirms this conjecture,
which plays a crucial role in the proof of Theorem \ref{degreeformula}.

\begin{theorem}
\label{main thm}Let $\mathbf{n}=(n_{0},n_{1},n_{2},n_{3})$ where $n_{k}\in
\mathbb{Z}_{\geq0}$ and $p\in E_{\tau}\backslash E_{\tau}[2]$. Then for each
$k=0,1,2,3$, we have
\begin{equation*}
\deg \sigma_{\mathbf{n,}p}=\deg \sigma_{\mathbf{n}_{k}^{+}}+\deg \sigma _{%
\mathbf{n}_{k}^{-}}.
\end{equation*}
\end{theorem}

Before giving the proof of Theorem \ref{main thm}, we would like to apply it
to prove Theorem \ref{degreeformula}.

\begin{proof}[Proof of Theorem \protect\ref{degreeformula}]
It was proved by Wang and the third author \cite{LW2} that%
\begin{equation}
\deg \sigma_{(n_{0},0,0,0)}=\frac{n_{0}(n_{0}+1)}{2},\text{ \ \ }n_{0}\in
\mathbb{Z}_{\geq0}.   \label{degn}
\end{equation}
Note that $(0,0,0,0)$ and $(-1,0,0,0)$ gives the same GLE, so (\ref{degn})
also holds for $n_{0}=-1$. Then by Theorem \ref{main thm}, we have
\begin{align}
\deg \sigma_{(n_{0},0,0,0),p} & =\deg \sigma_{(n_{0}+1,0,0,0)}+\deg
\sigma_{(n_{0}-1,0,0,0)}  \label{1111} \\
& =\frac{( n_{0}+1)( n_{0}+2) }{2}+\frac{(
n_{0}-1) n_{0}}{2}  \notag \\
& =n_{0}( n_{0}+1) +1.  \notag
\end{align}

Suppose that
\begin{equation*}
\deg \sigma_{(n_{0},k,0,0)}=\frac{n_{0}( n_{0}+1) }{2}+\frac{%
k(k+1)}{2}
\end{equation*}
for all $0\leq k\leq n_{1}-1$. We claim that
\begin{equation}
\deg \sigma_{(n_{0},n_{1},0,0)}=\frac{n_{0}( n_{0}+1) }{2}+\frac{%
n_{1}( n_{1}+1) }{2}.   \label{1114}
\end{equation}
We note that the following identities hold
\begin{align*}
& \deg \sigma_{(n_{0},n_{1},0,0)}+\deg \sigma_{(n_{0},n_{1}-2,0,0)} \\
& =\deg \sigma_{(n_{0},n_{1}-1,0,0),p} \\
& =\deg \sigma_{(n_{0}-1,n_{1}-1,0,0)}+\deg \sigma_{(n_{0}+1,n_{1}-1,0,0)},
\end{align*}
where the first identity follows from Theorem \ref{main thm} with $k=1$ and the second one follows from the $k=0$. Hence
{\allowdisplaybreaks
\begin{align}
&\quad \deg \sigma_{(n_{0},n_{1},0,0)}  \label{1115} \\
& =\deg \sigma_{(n_{0}-1,n_{1}-1,0,0)}+\deg
\sigma_{(n_{0}+1,n_{1}-1,0,0)}-\deg \sigma_{(n_{0},n_{1}-2,0,0)}  \notag \\
& =\frac{( n_{0}-1) n_{0}}{2}+\frac{( n_{1}-1) n_{1}}{2%
}+\frac{( n_{0}+1) ( n_{0}+2) }{2}+\frac{(
n_{1}-1) n_{1}}{2}  \notag \\
&\quad -\frac{n_{0}( n_{0}+1) }{2}-\frac{( n_{1}-2) (
n_{1}-1) }{2}  \notag \\
& =\frac{n_{0}( n_{0}+1) }{2}+\frac{n_{1}( n_{1}+1) }{2%
}.  \notag
\end{align}
}%
This proves (\ref{1114}).

Now we claim that
\begin{equation}
\deg \sigma_{(n_{0},n_{1},0,0),p}=n_{0}( n_{0}+1) +n_{1}(
n_{1}+1) +1.   \label{1116}
\end{equation}
Indeed, a direct consequence of (\ref{1114}) and Theorem \ref{main thm} imply that%
\begin{align}
& \deg \sigma_{(n_{0},n_{1},0,0),p}  \label{1117} \\
& =\deg \sigma_{(n_{0}+1,n_{1},0,0)}+\deg \sigma_{(n_{0}-1,n_{1},0,0)}
\notag \\
& =\frac{( n_{0}+1) ( n_{0}+2) }{2}+\frac {n_{1}(
n_{1}+1) }{2}+\frac{( n_{0}-1) n_{0}}{2}+\frac{n_{1}(
n_{1}+1) }{2}  \notag \\
& =n_{0}( n_{0}+1) +n_{1}(n_{1}+1) +1.  \notag
\end{align}
This proves (\ref{1116}).

By the same argument as (\ref{1115}) and (\ref{1117}), we could derive the
formulas (\ref{degh})-(\ref{degp}) for any $\mathbf{n}%
=(n_{0},n_{1},n_{2},n_{3})$ via induction.
\end{proof}

The rest of this section is devoted to the proof of Theorem \ref{main thm}
for $k=0$ (the other cases $k\in \{1,2,3\}$ can be proved in an analogous
way and we omit the details). Denote $\mathbf{n}_{0}^{\pm}=\mathbf{n}^{\pm}$
and%
\begin{equation*}
\deg \sigma_{\mathbf{n,}p}=m,\text{ }\deg \sigma_{\mathbf{n}^{+}}=m^{+},%
\text{ }\deg \sigma_{\mathbf{n}^{-}}=m^{-}.
\end{equation*}
Our goal is to prove $m=m^{+}+m^{-}$. Recalling Lemmas \ref{main thm4} and %
\ref{cor ind}, in the sequel we fix $\tau$ (so we will omit the notation $%
\tau$ freely) such that%
\begin{equation*}
\hat{a}_{m^{\pm}}(\tau)\not =0\text{ and }b_{m}(\cdot;\tau)\not \equiv 0.
\end{equation*}
Then $\wp(\sigma_{\mathbf{n}^{\pm}}(B,\hat{W}))$ (see (\ref{puu}) and (\ref%
{ab})) can be rewritten as%
\begin{equation}
\wp(\sigma_{\mathbf{n}^{\pm}}(B,\hat{W}))=C(\mathbf{n}^{\pm})\frac{\prod
_{i=1}^{m^{\pm}}( B-\hat{B}_{i}^{0\pm}) }{\prod_{i=1}^{m^{\pm}-1}%
( B-\hat{B}_{i}^{\infty \pm}) },   \label{pu5-5}
\end{equation}
where $\hat{B}_{i}^{0\pm}$'s and $\hat{B}_{i}^{\infty \pm}$'s are all the
zeros of $\hat{P}_{1}( B) $ and $\hat{P}_{2}( B) $
(corresponding to $\mathbf{n}^{\pm}$) respectively. On the other hand, since
$b_{m}(\cdot;\tau)\not \equiv 0$, we take any sequence $p\rightarrow0$ such
that $b_{m}(p)\not =0$ and $\wp(p)\not =0$, then $\wp( \sigma _{%
\mathbf{n,}p}( A,W) ) $ (see (\ref{pu}) and (\ref{a-b-k})) can be rewritten as%
\begin{equation}
\wp( \sigma_{\mathbf{n,}p}( A,W) ) =\wp (
p) \frac{\prod_{i=1}^{m}(A-A_{i}^{0}(p))}{\prod_{j=1}^{m}(A-A_{j}^{%
\infty}(p))},   \label{pu5}
\end{equation}
where $A_{i}^{0}(p)$'s and $A_{j}^{\infty}(p)$'s are all the zeros of $%
P_{1}(A,p)$ and $P_{2}(A,p)$ respectively.

To prove $m=m^{+}+m^{-}$, we need to study the asymptotics of $A_{i}^{0}(p)$%
's and $A_{j}^{\infty}(p)$'s as $p\rightarrow0$. By Corollary \ref{u0} and
Proposition \ref{prop11}, $A_{i}^{0}(p)$ satisfies the asymptotic formula (%
\ref{app}) as $p\rightarrow0$. We denote $A_{i}^{0}(p)$ by $%
A_{i}^{0^{+}}(p)$ if
\begin{equation}
A_{i}^{0^{+}}(p)=-( \tfrac{1}{4}+n_{0})
p^{-1}+\alpha_{i}^{0^{+}}p+o(p)\text{ for some }\alpha_{i}^{0^{+}}\in
\mathbb{C},   \label{q1}
\end{equation}
and by $A_{i}^{0^{-}}(p)$ if
\begin{equation}
A_{i}^{0^{-}}(p)=( \tfrac{3}{4}+n_{0})
p^{-1}+\alpha_{i}^{0^{-}}p+o(p)\text{ for some }\alpha_{i}^{0^{-}}\in
\mathbb{C}.   \label{q2}
\end{equation}
We assume that there are $A_{1}^{0^{+}}(p),\cdot \cdot
\cdot,A_{m_{1}}^{0^{+}}(p)$ and $A_{1}^{0^{-}}(p),\cdot \cdot
\cdot,A_{m_{2}}^{0^{-}}(p)$ (counted with multiplicity), so $m=m_{1}+m_{2}$.
Let $B_{i}^{0^{\pm}}(p)$ be the coefficient of GLE$(\mathbf{n,}%
p,A_{i}^{0^{\pm}}(p))$. Then by (\ref{qe}), we have%
\begin{equation*}
B_{i}^{0^{\pm}}(p)\rightarrow B_{i}^{0^{\pm}}\text{ as }p\rightarrow0,
\end{equation*}
where%
\begin{equation}
B_{i}^{0^{+}}:=-(2n_{0}+1)\alpha_{i}^{0^{+}}-%
\sum_{k=1}^{3}n_{k}(n_{k}+1)e_{k},   \label{q11}
\end{equation}%
\begin{equation}
B_{i}^{0^{-}}:=(2n_{0}+1)\alpha_{i}^{0^{-}}-%
\sum_{k=1}^{3}n_{k}(n_{k}+1)e_{k}.   \label{q12}
\end{equation}
Also by Proposition \ref{prop11}, GLE$(\mathbf{n,}p,A_{i}^{0^{\pm}}(p))%
\rightarrow$H$(\mathbf{n}^{\pm},B_{i}^{0^{\pm}})$ as $p\rightarrow0$. Since (%
\ref{pu5}) says $\wp(\sigma_{\mathbf{n,}p}(A_{i}^{0^{\pm}}(p),W_{i}^{0^{%
\pm}}(p)))=0$ for each $p$, we have%
\begin{equation*}
\wp(\sigma_{\mathbf{n}^{\pm}}(B_{i}^{0^{\pm}},W_{i}^{0^{\pm}}))=\lim
_{p\rightarrow0}\wp \left( \sigma_{\mathbf{n,}p}(A_{i}^{0^{%
\pm}}(p),W_{i}^{0^{\pm}}(p))\right) =0.
\end{equation*}
Recalling $\{ \hat{B}_{1}^{0^{+}},\cdot \cdot \cdot,\hat{B}_{m^{+}}^{0^{+}}\}
$ and $\{ \hat{B}_{1}^{0^{-}},\cdot \cdot \cdot,\hat{B}_{m^{-}}^{0^{-}}\}$
in (\ref{pu5-5})\textit{, }our above argument yields the following corollary.

\begin{corollary}
\label{cor11}$B_{i}^{0^{+}}$ $\in \{ \hat{B}_{1}^{0^{+}},\cdot \cdot \cdot ,%
\hat{B}_{m^{+}}^{0^{+}}\}$ for each $i$ $=1,2,\cdot \cdot \cdot,m_{1}$ and $%
B_{j}^{0^{-}}\in \{ \hat{B}_{1}^{0^{-}},\cdot \cdot \cdot,\hat{B}%
_{m^{-}}^{0^{-}}\}$ for each $j$ $=1,2,\cdot \cdot \cdot,m_{2}$.
\end{corollary}

Later we will prove $m_{1}=m^{+}$ and $m_{2}=m^{-}$. Indeed, our proof will imply $\{ B_{1}^{0^{+}},\cdot \cdot \cdot ,%
B_{m_1}^{0^{+}}\}=\{ \hat{B}_{1}^{0^{+}},\cdot \cdot \cdot ,%
\hat{B}_{m^{+}}^{0^{+}}\}$ and $\{ B_{1}^{0^{-}},\cdot \cdot \cdot ,%
B_{m_2}^{0^{-}}\}=\{ \hat{B}_{1}^{0^{-}},\cdot \cdot \cdot ,%
\hat{B}_{m^{-}}^{0^{-}}\}$. See (\ref{bbbb}) below.

Next we the asymptotics of $A_{i}^{\infty}(p)$'s as $p\rightarrow0$. Since (%
\ref{pu5}) says
\begin{equation}
\sigma_{\mathbf{n,}p}(A_{i}^{\infty}(p),W_{i}^{\infty}(p))=0\quad\text{ for each }%
p,   \label{0im}
\end{equation}
it might happen that the corresponding GLE$(\mathbf{n},p,A_{i}^{\infty}%
\left( p\right) )$ does not converge as $p\rightarrow0$. If GLE$(\mathbf{n}%
,p,A_{i}^{\infty}\left( p\right) )$ converges, again $A_{i}^{\infty}(p)$
satisfies the asymptotic formula (\ref{app}) as $%
p\rightarrow0$, and we denote it by $A_{i}^{\infty^{+}}(p)$ if
\begin{equation}
A_{i}^{\infty^{+}}(p)=-( \tfrac{1}{4}+n_{0}) p^{-1}+\alpha
_{i}^{\infty^{+}}p+o(p)\text{ for some }\alpha_{i}^{\infty^{+}}\in \mathbb{C}%
,   \label{q1-1}
\end{equation}
and by $A_{i}^{\infty^{-}}(p)$ if
\begin{equation}
A_{i}^{\infty^{-}}(p)=( \tfrac{3}{4}+n_{0}) p^{-1}+\alpha
_{i}^{\infty^{-}}p+o(p)\text{ for some }\alpha_{i}^{\infty^{-}}\in \mathbb{C}%
.   \label{q2-1}
\end{equation}
If GLE$(\mathbf{n},p,A_{i}^{\infty}\left( p\right) )$ does not converge,
then there are three alternatives of $A_{i}^{\infty}\left( p\right) $ up to
a subsequence, according to Proposition \ref{prop11}:

\begin{itemize}
\item[\textbf{(D-i)}] $A_{i}^{\infty}(p)=A_{i}^{\infty_{1}}(p)$ where
\begin{equation*}
A_{i}^{\infty_{1}}(p)=-( \tfrac{1}{4}+n_{0}) p^{-1}+\alpha
_{i}^{\infty_{1}}+o(1) \text{, }\alpha_{i}^{\infty_{1}}\not =0,
\end{equation*}

\item[\textbf{(D-ii)}] $A_{i}^{\infty}(p)=A_{i}^{\infty_{2}}(p)$ where%
\begin{equation*}
A_{i}^{\infty_{2}}(p)=( \tfrac{3}{4}+n_{0}) p^{-1}+\alpha
_{i}^{\infty_{2}}+o(1) \text{, }\alpha_{i}^{\infty_{2}}\not =0,
\end{equation*}

\item[\textbf{(D-iii})] $A_{i}^{\infty}(p)=A_{i}^{\infty_{3}}(p)$ such that
\begin{equation*}
A_{i}^{\infty_{3}}(p)p\rightarrow \beta \in \mathbb{C\cup \{ \infty \}}%
\backslash \{-(\tfrac{1}{4}+n_{0}),(\tfrac{3}{4}+n_{0})\}.
\end{equation*}
\end{itemize}

The key step is to prove the following result.

\begin{lemma}
\label{lemma5} Recall (\ref{0im}) that $\sigma_{\mathbf{n,}p}^{-1}\left(
0\right) $ $=\{ \left( A_{i}^{\infty}(p),W_{i}^{\infty}(p)\right)
\}_{i=1}^{m}$. Then \textit{there are }$m_{1}-1$\textit{\ of }$\{$\textit{GLE%
}$\mathit{(}\mathbf{n},p,A_{i}^{\infty}(p)\mathit{)}\}_{i=1}^{m}$\textit{\
converge to H}$\mathit{(}\mathbf{n}^{+},B\mathit{)}$ for some $B$\textit{,
and }$m_{2}-1$\textit{\ of them converge to H}$\mathit{(}\mathbf{n}^{-},B%
\mathit{)}$\textit{\ for some }$B$\textit{, and\ the rest two (say }$%
A_{m-1}^{\infty}(p)$\textit{\ and }$A_{m}^{\infty}(p)$\textit{)} \textit{do
not converge. Moreover, }$A_{m-1}^{\infty}(p)$\textit{\ and }$%
A_{m}^{\infty}(p)$ \textit{belong to the case} \textbf{(D-iii)}\textit{\ with%
} $A_{i}^{\infty}(p)p\rightarrow \beta_{i}\in \mathbb{C}\backslash \{-(%
\tfrac{1}{4}+n_{0}),(\tfrac{3}{4}+n_{0})\}$, $i=m-1,m$ \textit{as} $%
p\rightarrow0$.
\end{lemma}

\begin{proof}
\textbf{Step 1}. We prove that there are some $i$'s such that GLE$(\mathbf{n}%
,p,A_{i}^{\infty}(p))$ does not converge as $p\rightarrow0$.

Suppose GLE$(\mathbf{n},p,A_{i}^{\infty}\left( p\right) )$ are convergent
for all $i$. Suppose $k_{1}$ of them converge to H$(\mathbf{n}^{+},B)$ and $%
k_{2}$ of them converge to H$(\mathbf{n}^{-},B)$. Then we rewrite (\ref{pu5}%
) as
\begin{equation}
\wp \left( \sigma_{\mathbf{n,}p}\left( A,W\right) \right) =\wp (p)\frac{%
\prod_{i=1}^{m_{1}}(A-A_{i}^{0^{+}}(p))\cdot
\prod_{i=1}^{m_{2}}(A-A_{i}^{0^{-}}(p))}{\prod_{i=1}^{k_{1}}\left(
A-A_{i}^{\infty^{+}}(p)\right) \cdot \prod_{i=1}^{k_{2}}\left(
A-A_{i}^{\infty^{-}}(p)\right) }.   \label{x}
\end{equation}
Fix $\sigma_{0}\in E_{\tau}\backslash E_{\tau}[2]$ such that $\wp(
\sigma_{0}) \not \in \{0,\infty \}$. The same proof of Theorem \protect\ref{cor7} implies that
there exist $\sigma_{p}^{\pm}\in E_{\tau}\backslash E_{\tau}[2]$ and
\begin{equation}
(A^{\pm}(p),W^{\pm}(p))\in \sigma_{\mathbf{n,}p}^{-1}(\sigma_{p}^{\pm})\;\text{with}\; \lim_{p\to 0}\sigma_p^{\pm}=\sigma_0,
\label{plus-1}
\end{equation}
such that GLE$(\mathbf{n},p,A^{\pm}(p))$ converges to H$(\mathbf{n}^{\pm
},B^{\pm})$ for some $B^{\pm}$, namely%
\begin{equation}
A^{+}(p)=-\left( \tfrac{1}{4}+n_{0}\right) p^{-1}+\alpha^{+}p+o(p)\text{ for
some }\alpha^{+}\in \mathbb{C},   \label{plus}
\end{equation}%
\begin{equation}
A^{-}(p)=\left( \tfrac{3}{4}+n_{0}\right) p^{-1}+\alpha^{-}p+o(p)\text{ for
some }\alpha^{-}\in \mathbb{C}.   \label{plus1}
\end{equation}
Recalling (\ref{q1})-(\ref{q2}) and (\ref{q1-1})-(\ref{q2-1}), we remark
that
\begin{equation}
\alpha^{+}\not \in \{ \alpha_{i}^{0^{+}}\}_{i}\cup \{
\alpha_{j}^{\infty^{+}}\}_{j},\text{ \ }\alpha^{-}\not \in \{
\alpha_{i}^{0^{-}}\}_{i}\cup \{ \alpha_{j}^{\infty^{-}}\}_{j},
\label{plus2}
\end{equation}
namely $A^{\pm}(p)-A_{i}^{0^{\pm}}(p)=(\alpha^{\pm}-\alpha_{i}^{0^{%
\pm}})p+o(p)$ with $\alpha^{\pm}-\alpha_{i}^{0^{\pm}}\not =0$ for any $i$
and so do $A^{\pm}(p)-A_{i}^{\infty^{\pm}}(p)$. For example, if $%
\alpha^{+}=\alpha_{i}^{0^{+}}$ for some $i$, then GLE$(\mathbf{n}%
,p,A_{i}^{0^{+}}\left( p\right) )$ also converges to H$(\mathbf{n}^{+},B^{+})
$, the same limit as that of GLE$(\mathbf{n},p,A^{+}\left( p\right) )$. This yields
from (\ref{plus-1}) that%
\begin{align*}
0 & =\wp \left( \sigma_{\mathbf{n,}p}\left( A_{i}^{0^{+}}( p)
,W_{i}^{0^{+}}( p) \right) \right) =\wp \left( \sigma _{\mathbf{n}%
^{+}}\left( B^{+},W^{+}\right) \right) \\
& =\lim_{p\to 0}\wp \left( \sigma_{\mathbf{n,}p}\left( A^{+}( p) ,W^{+}(
p) \right) \right)=\lim_{p\to 0}\wp(\sigma_p^{+}) =\wp(\sigma_{0})\not \in \{0,\infty \},
\end{align*}
a contradiction.

Now by inserting $A$ $=A^{+}(p)$ into (\ref{x}), we easily see from (\ref{q1}%
)-(\ref{q2}), (\ref{q1-1})-(\ref{q2-1}) and (\ref{plus}) that
\begin{equation}
\wp(\sigma_{0})+o(1)=\wp(\sigma_p^{+})=\frac{p^{m_{1}-m_{2}-2}\cdot \mathcal{O}(1)}{%
p^{k_{1}-k_{2}}\cdot \mathcal{O}(1)}\text{ as }p\rightarrow0.   \label{d+}
\end{equation}
Here different from the notation $O(1)$, we use the notation $\mathcal{O}(1)$
to denote various quantities depending on $p$ which is uniformly bounded
away from $0$ and $\infty$ as $p\rightarrow0$. However, inserting $A$ $%
=A^{-}(p)$ into (\ref{x}) leads to%
\begin{equation}
\wp(\sigma_{0})+o(1)=\wp(\sigma_p^{-}) =\frac{p^{m_{2}-m_{1}-2}\cdot \mathcal{O}(1)}{%
p^{k_{2}-k_{1}}\cdot \mathcal{O}(1)},   \label{d-}
\end{equation}
which contradicts with (\ref{d+}). This proves Step 1, namely there must
exist
\begin{equation*}
A_{i}^{\infty}( p) \in \mathbf{A}^{\infty}:= \{
A_{1}^{\infty}( p) ,\text{ }\cdot \cdot \cdot \text{ }%
,A_{m}^{\infty }( p) \}
\end{equation*}
such that GLE($\mathbf{n},p,A_{i}^{\infty}\left( p\right) $) does not
converge. Suppose there are $\ell_{1}$, $\ell_{2}$ and $\ell_{3}$ $%
A_{i}^{\infty}\left( p\right) $'s satisfying (D-i), (D-ii) and (D-iii)
respectively. Then we rewrite (\ref{pu5}) as%
\begin{align}
& \wp \left( \sigma_{\mathbf{n,}p}\left( A,W\right) \right) =  \label{xx} \\
& \frac{\wp \left( p\right) \prod_{i=1}^{m_{1}}(A-A_{i}^{0^{+}}(p))\cdot
\prod_{i=1}^{m_{2}}(A-A_{i}^{0^{-}}(p))}{\prod_{i=1}^{k_{1}}(A-A_{i}^{%
\infty^{+}}(p))\prod_{i=1}^{k_{2}}(A-A_{i}^{\infty^{-}}(p))\prod_{j=1}^{3}%
\prod_{i=1}^{\ell_{j}}(A-A_{i}^{\infty_{j}}(p))}.  \notag
\end{align}

\textbf{Step 2}. We prove that $\left( \ell_{1},\ell_{2},\ell_{3}\right) \in
\{ \left( 0,0,2\right) ,\left( 0,2,1\right) ,\left( 2,0,1\right) \}$.

By the asymptotics (\ref{plus})-(\ref{plus1}) and (D-i)-(D-iii), we have%
\begin{equation}
\left \{
\begin{array}{l}
A^{+}(p)-A_{i}^{\infty_{1}}(p)=-\alpha_{i}^{\infty_{1}}+o( 1) ,%
\text{ }\alpha_{i}^{\infty_{1}}\not =0, \\
A^{+}(p)-A_{i}^{\infty_{2}}(p)=-(1+2n_{0})p^{-1}+O(1),%
\end{array}
\right.   \label{a11}
\end{equation}%
\begin{equation}
\left \{
\begin{array}{l}
A^{-}(p)-A_{i}^{\infty_{1}}(p)=(1+2n_{0})p^{-1}+O(1), \\
A^{+}(p)-A_{i}^{\infty_{2}}(p)=-\alpha_{i}^{\infty_{2}}+o( 1) ,%
\text{ }\alpha_{i}^{\infty_{2}}\not =0,%
\end{array}
\right.   \label{a12}
\end{equation}
and%
\begin{equation}
\prod_{i=1}^{\ell_{3}}\frac{A^{+}( p) -A_{i}^{\infty_{3}}(
p) }{A^{-}( p) -A_{i}^{\infty_{3}}( p) }=%
\mathcal{O}(1).   \label{o}
\end{equation}
Again by inserting $A=A^{\pm}(p)$ into (\ref{xx}) and using (\ref{plus-1}), we have the following
identity%
\begin{align}
\wp(\sigma_{0}) & =\frac{p^{m_{1}-m_{2}-2}\cdot \mathcal{O}(1)}{%
p^{k_{1}-k_{2}-\ell_{2}}\prod_{i=1}^{\ell_{3}}(A^{+}(p)-A_{i}^{%
\infty_{3}}(p))}+o(1)  \label{xxx} \\
& =\frac{p^{m_{2}-m_{1}-2}\cdot \mathcal{O}(1)}{p^{k_{2}-k_{1}-\ell_{1}}%
\prod_{i=1}^{\ell_{3}}(A^{-}(p)-A_{i}^{\infty_{3}}(p))}+o(1).  \notag
\end{align}
Together with (\ref{o}), we obtain
\begin{equation}
2\left( m_{1}-m_{2}\right) =2\left( k_{1}-k_{2}\right) +\left( \ell
_{1}-\ell_{2}\right) ,   \label{m11}
\end{equation}
which implies $\ell_{1}-\ell_{2}$ is even. Since $\wp(\sigma_{0})\not =0$, (%
\ref{xxx}) and $|A^{+}(p)-A_{i}^{\infty_{3}}(p)|\geq|p|^{-1}\cdot \mathcal{O}%
(1)$ also yield
\begin{align}
\left \vert p\right \vert ^{-\ell_{3}} & \leq \mathcal{O}(1)\cdot\prod
\limits_{i=1}^{\ell_{3}}|A^{+}(p)-A_{i}^{\infty_{3}}(p)|  \label{m12} \\
& =\mathcal{O}(1)\cdot|p|^{m_{1}-m_{2}-2-(k_{1}-k_{2}-\ell_{2})}=\mathcal{O}%
(1)\cdot|p|^{\frac{\ell_{1}+\ell_{2}}{2}-2},  \notag
\end{align}
which implies $\frac{\ell_{1}+\ell_{2}}{2}-2\leq-\ell_{3}$, namely%
\begin{equation}
0\leq \ell_{1}+\ell_{2}+2\ell_{3}\leq4.   \label{m13}
\end{equation}
From here and that $\ell_{1}-\ell_{2}$ is even, we see that
\begin{align*}
\left( \ell_{1},\ell_{2},\ell_{3}\right) & =\left( 0,0,1\right) ,\left(
0,0,2\right) ,\left( 0,2,0\right) ,\left( 0,2,1\right) , \\
& \left( 2,0,0\right) ,\left( 2,0,1\right) ,\left( 1,1,0\right) ,\left(
1,1,1\right), (1,3,0), (3,1,0).
\end{align*}
On the other hand, we have%
\begin{equation}
m_{1}+m_{2}=m=k_{1}+k_{2}+\ell_{1}+\ell_{2}+\ell_{3}.   \label{m14}
\end{equation}
From here and (\ref{m11}), we can only have%
\begin{equation*}
\left( \ell_{1},\ell_{2},\ell_{3}\right) =\left( 0,0,2\right) ,\left(
0,2,1\right) ,\left( 2,0,1\right) ,\left( 1,1,0\right) .
\end{equation*}
The case $\left( \ell_{1},\ell_{2},\ell_{3}\right) $ $=\left( 1,1,0\right) $
is impossible by (\ref{xxx}). This proves Step 2.

\textbf{Step 3}. We prove that $(\ell_{1},\ell_{2},\ell_{3})\not
=(0,2,1),(2,0,1)$ and so $(\ell_{1},\ell_{2},\ell_{3})=(0,0,2)$.

Suppose $( \ell_{1},\ell_{2},\ell_{3}) =( 0,2,1) $.
Then it follows from (\ref{m11}) and (\ref{m14}) that $k_{1}=m_{1}-1$, $%
k_{2}=m_{2}-2$. Inserting these into (\ref{xxx}) leads to $%
A^{\pm}(p)-A_{1}^{\infty_{3}}(p)=\mathcal{O}(1)\cdot p^{-1}$, so
\begin{equation}
A_{1}^{\infty_{3}}(p)p\rightarrow \beta_{1}\in \mathbb{C}\backslash \{-(\tfrac{1}{4}%
+n_{0}),\tfrac{3}{4}+n_{0}\}.   \label{infty3}
\end{equation}

To get a contradiction, we take any $B$ and consider $A(p)$ such that GLE$(%
\mathbf{n},p,A(p))\rightarrow$ H$(\mathbf{n}^{-},B)$ as $p\rightarrow0$. The
existence of such $A(p)$ is proved in Theorem \ref{cor7}, which has the
following asymptotics%
\begin{equation*}
A(p)=\left( \tfrac{3}{4}+n_{0}\right) p^{-1}+\alpha_{1}^{-}p+o(
p)
\end{equation*}
where%
\begin{equation*}
\alpha_{1}^{-}:=\frac{B+\sum_{k=1}^{3}n_{k}(
n_{k}+1) e_{k}}{2n_{0}+1} .
\end{equation*}
Furthermore, the proof of Theorem \ref{cor7} also implies $\wp(\sigma_{\mathbf{n,}p}\left(
A(p),W(p)\right) )\rightarrow \wp(\sigma_{\mathbf{n}^{-}}(B,\hat{W}))$. By (\ref%
{xx}) and $( \ell_{1},\ell_{2},\ell_{3}) =( 0,2,1) $ we have
\begin{align*}
& \wp(\sigma_{\mathbf{n,}p}\left( A(p),W(p)\right) )= \\
& \frac{\wp(p)
\prod_{i=1}^{m_{1}}(A(p)-A_{i}^{0^{+}}(p))%
\prod_{i=1}^{m_{2}}(A(p)-A_{i}^{0^{-}}(p))(A(p)-A_{1}^{\infty_{3}}(p))^{-1}}{%
\prod_{i=1}^{m_{1}-1}(A(p)-A_{i}^{\infty^{+}}(p))%
\prod_{i=1}^{m_{2}-2}(A(p)-A_{i}^{\infty^{-}}(p))%
\prod_{i=1}^{2}(A(p)-A_{i}^{\infty_{2}}(p))}.
\end{align*}
Inserting the asymptotics of $A( p) $, (\ref{q1})-(\ref{q2}), (%
\ref{q1-1})-(\ref{q2-1}), (D-2) and (\ref{infty3}) to the above formula and
letting $p\rightarrow0$, we easily obtain%
\begin{equation}
\wp(\sigma_{\mathbf{n}^{-}}(B,\hat{W}))=\frac{\prod_{i=1}^{m_{2}}(
B-B_{i}^{0^{-}}) }{(1+2n_{0})(\frac{3}{4}+n_{0}-\beta_{1})\alpha_{1}^{%
\infty_{2}}\alpha_{2}^{\infty_{2}}\prod_{i=1}^{m_{2}-2}(
B-B_{i}^{\infty^{-}}) },   \label{***}
\end{equation}
where $B_{i}^{0^{-}}=\lim_{p\rightarrow0}B_{i}^{0^{-}}(p)$ is given in (\ref%
{q12}) and similarly for $B_{i}^{\infty^{-}}=\lim_{p\rightarrow0}B_{i}^{%
\infty^{-}}(p)$. However, (\ref{***}) contradicts with (\ref{puu}) and (\ref%
{113}). This proves $\left( \ell_{1},\ell_{2},\ell_{3}\right) \not =\left(
0,2,1\right) $. Similarly we can prove $\left(
\ell_{1},\ell_{2},\ell_{3}\right) \not =\left( 2,0,1\right) $ and so $(\ell
_{1},\ell_{2},\ell_{3})=(0,0,2)$.

\textbf{Step 4}. We complete the proof.

Since $\left( \ell_{1},\ell_{2},\ell_{3}\right) $ $=\left( 0,0,2\right) $,
then (\ref{m11}) and (\ref{m14}) imply $k_{1}=m_{1}-1$ and $k_{2}=m_{2}-1$.
Inserting these into (\ref{xxx}) leads to $\prod_{i=1}^{2}(A^{%
\pm}(p)-A_{i}^{\infty_{3}}(p))=\mathcal{O}(1)\cdot p^{-2}$ and so $%
A_{i}^{\infty_{3}}(p)p\rightarrow \beta_{i}\in \mathbb{C}\backslash \{-(%
\tfrac{1}{4}+n_{0}),(\tfrac{3}{4}+n_{0})\}$. The proof is complete.
\end{proof}

We are in the position to prove Theorem \ref{main thm}.

\begin{proof}[Proof of Theorem \protect\ref{main thm}]
We prove the case $k=0$ (the other cases $k\in \{1,2,3\}$ are similar). As
pointed out before, we only need to prove $m^{+}=m_{1}$ and $m^{-}=m_{2}$.

Since by Lemma \ref{lemma5}, we can rewrite (\ref{pu5}) as
\begin{align*}
& \wp \left( \sigma_{\mathbf{n,}p}\left( A,W\right) \right) \\
& =\frac{\wp(p)\prod_{i=1}^{m_{1}}(A-A_{i}^{0^{+}}(p))\cdot
\prod_{i=1}^{m_{2}}(A-A_{i}^{0^{-}}(p))}{\prod_{i=1}^{m_{1}-1}\left(
A-A_{i}^{\infty ^{+}}(p)\right) \prod_{i=1}^{m_{2}-1}\left(
A-A_{i}^{\infty^{-}}(p)\right) \prod_{i=m-1}^{m}\left(
A-A_{i}^{\infty}(p)\right) }.
\end{align*}
Then by repeating the argument of Step 3 in Lemma \ref{lemma5}, we obtain%
\begin{equation}\label{bbbb}
\wp(\sigma_{\mathbf{n}^{-}}(B,\hat{W}))=\hat{C}(\mathbf{n}^{-})\frac {%
\prod_{i=1}^{m_{2}}(B-B_{i}^{0^{-}})}{\prod_{i=1}^{m_{2}-1}(
B-B_{i}^{\infty^{-}}) },
\end{equation}
where $\hat{C}(\mathbf{n}^{-})$ is a nonzero constant. Comparing this
with (\ref{pu5-5}), we conclude $m_{2}=m^{-}=\deg \sigma_{\mathbf{n}^{-}}$.
Similarly we can prove $m_{1}=m^{+}=\deg \sigma_{\mathbf{n}^{+}}$. In
conclusion,%
\begin{equation*}
\deg \sigma_{\mathbf{n}^{+}}+\deg \sigma_{\mathbf{n}^{-}}=m_{1}+m_{2}=\deg
\sigma_{\mathbf{n},p}.
\end{equation*}
This completes the proof.
\end{proof}

\bigskip

{\bf Acknowledgements} The research of the first author was supported by NSFC.

\end{document}